\documentclass{mainresearch}

\title{Decompositions of Tetrahedra into Height Functions}

\author{%
	Connor James Castillo%
	\and%
	Jackson Bradley Smith%
}


\begin{document}
	
\maketitle
\abstract{In short, we call a polyhedron $P \subset \mathbb{R}^3$ a height function if there exists some face $F$ of $P$ such that $P$ lies entirely within $F$ when orthogonally projected onto the affine plane containing $F$. Because  tetrahedra come in various configurations, we propose a dihedral angle classification of tetrahedra, and prove that a tetrahedron $T$ is either a height function already, or, at most one planar bisection of $T$ is needed to produce two height function tetrahedra. Furthermore, we prove that such bisections belong to infinitely large families of bisections.}
\section{Introduction}
When 3-D printing objects with features that are either geometrically simple or complex, a bottom-up process is used to print each layer. In doing so, the printing process is executed in a way that, in theory, ensures the structural stability required for printing layers that lie higher above the \emph{build plate}. While this process is relatively simple for layers that lie directly above the first layer of the print, it becomes much more complicated for features of the object that protrude past the first layer.
\begin{figure}[h!]
	\centering
	\includegraphics[scale=.21]{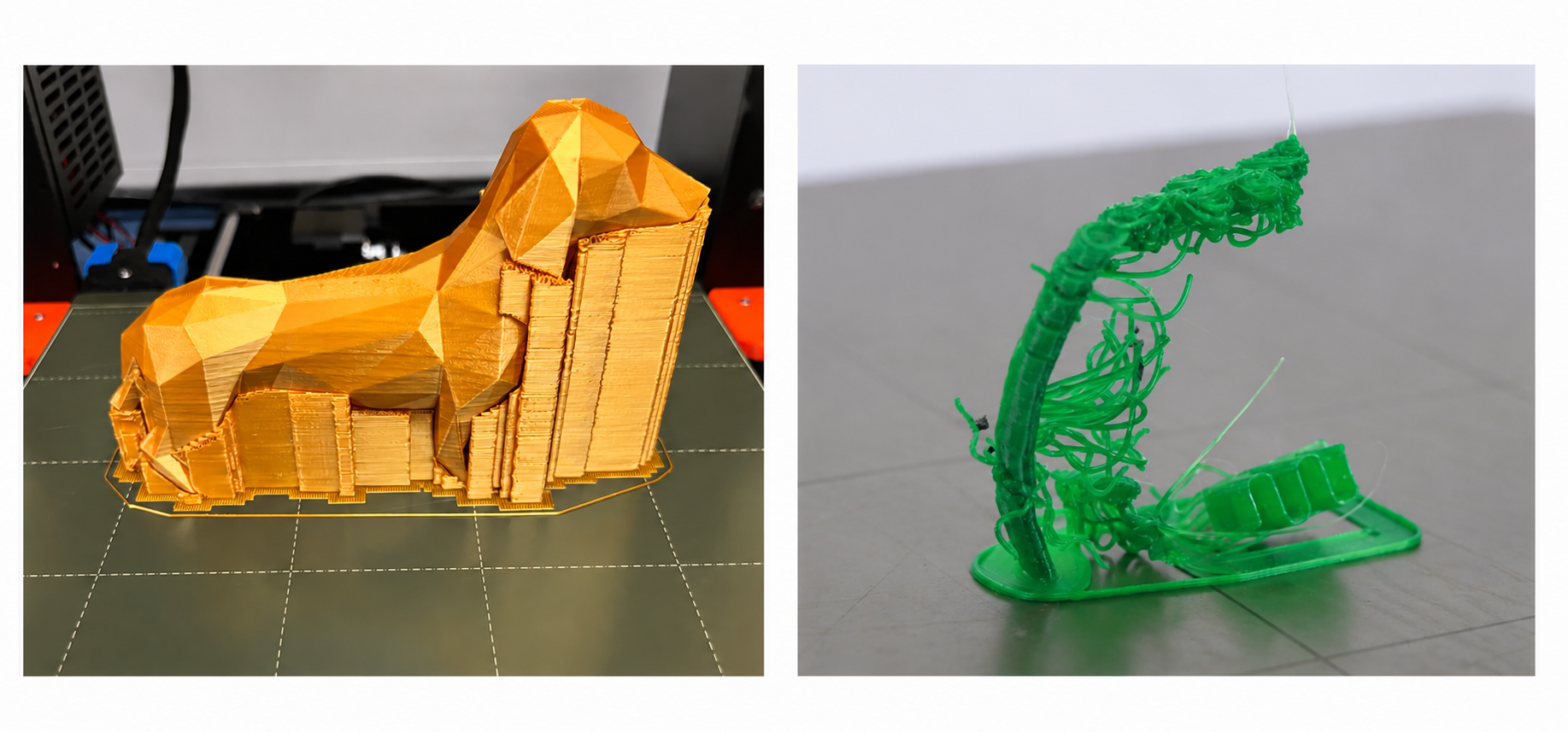}
	\caption{Example of successful support generation for a low-poly dachshund model \cite{printables_make397148} alongside an example of failed support structures during printing \cite{prusasupports}.}
	\label{fig:placeholder}
\end{figure}\label{dog}
In particular, additional \emph{supports} must be printed to ensure that the overhanging features do not collapse during the printing process. 

Though the use of supports can be, at times, effective for printing overhanging features, they open up the possibility for printing failures as seen in Figure \ref{dog}. For objects that have a significantly large amount of overhanging features due to their complexity, this issue becomes more likely. As a solution, methods of object decomposition can be employed to reduce the number of supports required to print objects with overhanging features. However, any amount of decomposition comes at the cost of increasing the total number of parts, which, in turn, may increase the total points of failure for objects being used in practical applications involving any amount of physical stress. Thus, it is within our best interest to search for decompositions that minimize both the total number of objects after decomposition and the number of supports required to print all pieces. 

In this paper, we study tetrahedra in $\mathbb{R}^3$, with particular emphasis on the relationship between their projective properties and geometric decompositions. Our primary objective is to understand when a tetrahedron admits a height function structure and how planar bisections can be used to produce such structures. By combining geometric, combinatorial, and projective techniques, we obtain a classification of tetrahedra and a complete description of the subdivisions that arise from planar cuts.

In section \ref{PRELIM}, we introduce relevant terminology with heavy emphasis on developing mathematical descriptions of supports. In particular, we define height function polyhedra and develop the geometric intuition necessary to study them in the context of tetrahedra. We also introduce planar bisections and discuss the special case of tetrahedral subdivisions, which form the principal class of decompositions considered in this work. These concepts establish the framework upon which the later classification results are built.

In section \ref{anglefunction}, we introduce the concept of a parametrized family of tetrahedral subdivisions on a fixed tetrahedron, and discuss the application of discrete angle functions on both fixed tetrahedral subdivisions and parametrized families of subdivisions.

In section \ref{tetrahedra}, we develop a classification of tetrahedra based on the number and arrangement of their obtuse dihedral angles. This classification significantly reduces the complexity of the problem by organizing all tetrahedra into a finite collection of geometric types. We then establish a height function criterion relating the dihedral angle structure of a tetrahedron to its projective properties.

In section \ref{BISEC}, we prove that every tetrahedron that is not already a height function can be transformed into one through at most a single tetrahedral subdivision. Furthermore, we show that any such subdivision is unique in type, yielding an almost complete description of affine planes that induce height function tetrahedral subdivisions of tetrahedra. 

Finally, in section \ref{SEC5}  , we unify the results of sections 4 and 5, and classify all height function tetrahedral subdivisions in terms of the families of affine planes that induce them. We will show that these families are infinite in size and unique up to isomorphisms.

\section{Polyhedra, Height Functions, and Planar Bisections}\label{PRELIM}
\subsection{Polyhedra and Height Functions}
Generally speaking, representations of polyhedra appear in a wide range of mathematical contexts, including but not limited to combinatorics, simplicial geometry, and of course, convex geometry. In the convex setting, we formally define a polyhedron $P$ as a convex set in $\mathbb{R}^n$ obtained as the convex hull of finitely many points with finitely many rays. Within this paper, we specifically focus on tetrahedra, that is, polyhedra formed by four affinely independent points in $\mathbb{R}^3$. More concretely, for the entirety of this paper, unless otherwise specified, let us define any tetrahedron $T$ 
as
\begin{equation*}
	T = \operatorname{conv} \{v_0,v_1,v_2,v_3\}, \quad v_i \in \mathbb{R}^3
\end{equation*}
For a tetrahedron $T$, we denote $V(T),E(T)$, and $F(T)$ as the sets of vertices, edges, and faces of $T$ respectively.

Given that we aim to analyze the concept of 3-Dimensional supports in a mathematical setting relating to $T$, we would naturally like to focus on what necessarily determines a feature to be "overhanging" in $\mathbb{R}^3$ relative to some face $F$ of $T$. Conveniently, we may use the orthogonal projection onto the affine plane containing the face $F$ to formalize this notion. Intuitively, a point of $T$ may be regarded as non-overhanging relative to a face $F$ whenever its orthogonal projection onto the affine plane containing $F$ lies within $F$ itself. This motivates the following definition, which captures the idea that every point of the tetrahedron is vertically supported by the face $F$.

\begin{figure}[h!]
	\centering
	\includegraphics[scale=.70]{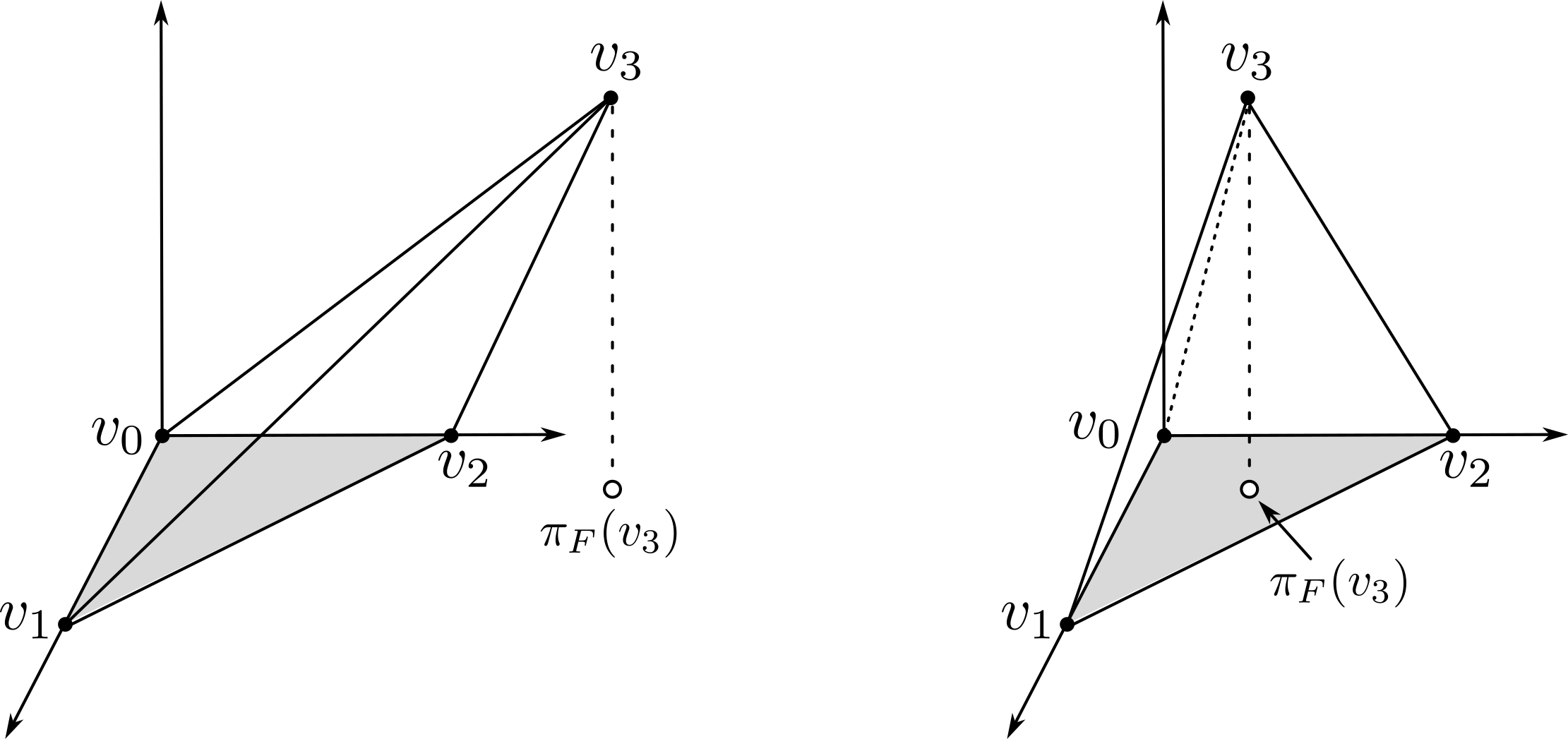}
	\caption{(\textit{Left}) A non height function tetrahedron $T \subset \mathbb{R}^3$. (\textit{Right}) A height function tetrahedron $T \subset \mathbb{R}^3$ with a foundation $F$, highlighted in gray.} 
	\label{HEIGHT}
\end{figure}
\begin{defin}
	A polyhedron $P$ is called a \emph{height function} if there exists a face 
	$F \in F(P)$ such that the orthogonal projection 
	$\pi_F \colon \mathbb{R}^3 \to \mathrm{aff}(F)$ satisfies
	\[
	\pi_F(P) \subseteq F.
	\]
	In which case, we call the face $F$ a \emph{foundation} of $P$.
\end{defin}

As we will discuss later in section \ref{BISEC}, it may be the case that a tetrahedron is not naturally a height function over any of its faces. In that case, we aim to decompose $T$ in the interest of producing sub-tetrahedra that are height functions over their respective foundation faces.

\subsection{Planar Bisections}
While there exists a wide variety of decomposition methods, we opt for a simple yet effective method that complies with the geometric constraints imposed by the convexity of $T$.
\begin{defin}\label{PBISEC}
	A planar bisection of a polyhedron $P \subset \mathbb{R}^3$ is a decomposition
	\begin{equation}\label{E1}
		P = P_1 \cup P_2
	\end{equation}
	such that there exists an affine plane $\Pi \subset \mathbb{R}^3$ satisfying
	\begin{equation}\label{E2}
		P_1 \cap P_2 = P \cap \Pi
	\end{equation}  
	where $P_1$ and $P_2$ are convex polyhedra contained in the two closed half-spaces determined by $\Pi$.
\end{defin}
\begin{remark}
	In the proof for Lemma \ref{TS}, we will explicitly use half-spaces. However, for future proofs, we will directly define the decomposition.
\end{remark}
Intuitively, we may view the affine plane $\Pi$ as a bisecting mechanism whose intersection with $P$ defines a common face to be shared by the resulting polyhedra $P_1$ and $P_2$. Note, however, that the decomposition described in Definition \ref{PBISEC} imposes no restrictions on the combinatorial or geometric type of the resulting polyhedra $P_1$ and $P_2$. In particular, if the affine plane $\Pi$ coincides with a face $F$ of $P$, then the induced decomposition fails to produce two genuine 3-dimensional sub-polyhedra of $P$. When this occurs, one sub-polyhedron of $P$ is identified with $P$ itself, while the other collapses to a 2-dimensional polygon corresponding to a face of $P$. For the purpose of this study, we regard any such planar bisection as \textit{degenerate}, and refer to the resulting collapse of $P_i$ as a \textit{degeneration} of the polyhedron.
\begin{remark}
	Strictly speaking, degenerate bisections do not satisfy the requirements of a planar bisection. We nevertheless refer to them as such when the meaning is clear from context.
\end{remark}

The existence of degenerate bisections can be found in all types of polyhedra, including tetrahedra, hinting at an important dichotomy present in the family of all affine planes $\Pi$ that induce planar bisections on polyhedra. In the context of tetrahedra, we focus on a very specific sub-family of affine planes that induce planar bisections of a more favorable type for our study.

\begin{defin}\label{2.3}
	Let $T \subset \mathbb{R}^3$ be a tetrahedron. A \textit{tetrahedral subdivision} of $T$ is a planar bisection
	\begin{equation*}
		T = T_1 \cup T_2
	\end{equation*}
	such that $T_1 \cap T_2 = T \cap \Pi$ for some affine plane $\Pi \subset \mathbb{R}^3$, where $T_1$ and $T_2$ are distinct $3$-dimensional tetrahedra.
\end{defin}

\begin{figure}[h!]
	\centering
	\includegraphics[scale=.7]{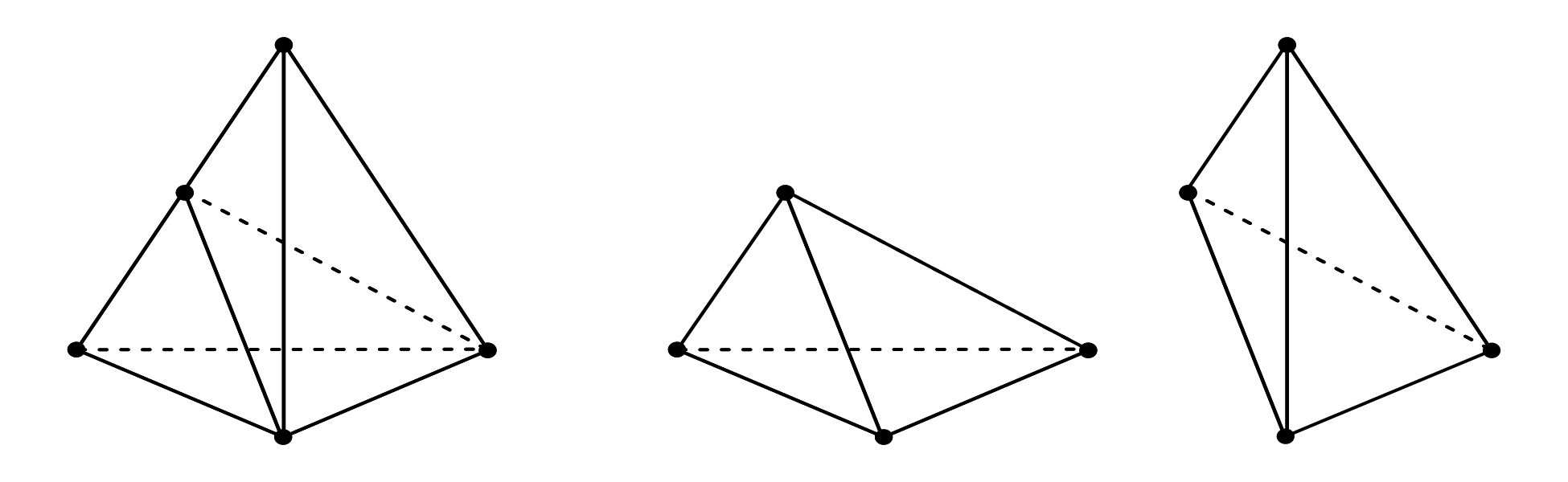}
	\caption{Tetrahedral subdivision of a tetrahedron $T$ with sub-tetrahedra separated.}
\end{figure}
\begin{figure}[h!]
	\centering
	\includegraphics[scale=.7]{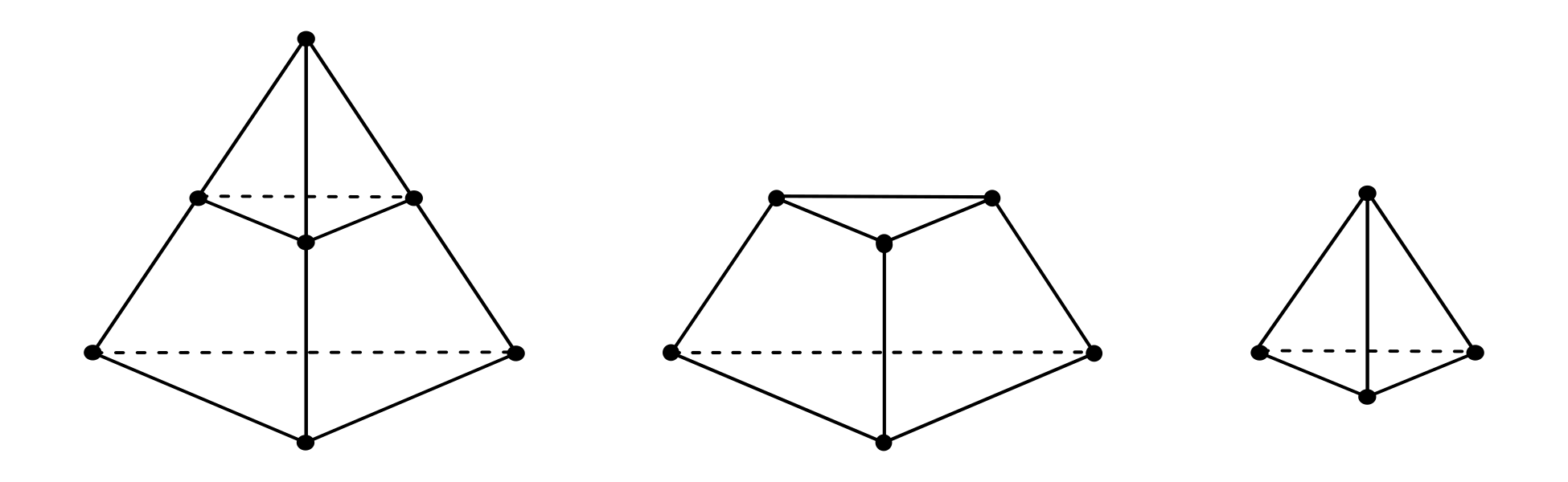}
	\caption{Non-tetrahedral subdivision planar bisection of a tetrahedron $T$ with sub-polyhedra separated.}
\end{figure}

Since the primary focus of this paper is to analyze height function tetrahedra, we adopt an informal definition of a specialized tetrahedral subdivision. We call a tetrahedral subdivision a \textit{height function tetrahedral subdivision} if the resulting sub-tetrahedra $T_1$ and $T_2$ are both height functions.

The use of Definition \ref{2.3}, similar to degenerate decompositions, directly implies that there exists affine planes whose intersection with $T$ produces a non-tetrahedral subdivision of $T$. As a conclusion to this section, let us move to formally classify the conditions for which any affine plane $\Pi$ must satisfy to induce a tetrahedral subdivision on $T$.

\begin{lemma}\label{TS}
	Let $T \subset \mathbb{R}^3$ be a tetrahedron, and let $\Pi$ be an affine plane. Then $\Pi$ induces a tetrahedral subdivision of $T$ if and only if $\Pi$ contains an edge $e \in E(T)$ and intersects the opposite edge $e' \in E(T)$ at exactly one point $p \in \operatorname{relint}(e')$.
\end{lemma}
\begin{proof}
	$(\Rightarrow)$ Let $T = \operatorname{conv}\{v_0,v_1,v_2,v_3\} \subset \mathbb{R}^3$ be a tetrahedron, and assume $\Pi$ induces a tetrahedral subdivision of $T$. By definition,
	\[
	T_0=T\cap H^- \quad \text{ and } \quad T_1=T\cap H^+
	\]
	are both tetrahedra where $H^-$ and $H^+$ are the closed half-spaces determined by $\Pi$. Thus, the face they share, defined by $T \cap \Pi$, is triangular and can be described as,
	\[
	T \cap \Pi = \operatorname{conv}\{p_1,p_2,p_3\} \quad p_i\in \mathbb{R}^3
	\]
	Because $T_0$ and $T_1$ are both tetrahedra, we may express them in terms of their 4 vertices, 3 of which they have in common through the face $T \cap \Pi$:
	\[
	T_0=\operatorname{conv}\{v_i,p_1,p_2,p_3\} \quad \text{ and } \quad T_1 = \operatorname{conv}\{v_j,p_1,p_2,p_3\}
	\]
	Notice that for any non-degenerate tetrahedral subdivision of $T$, every vertex of $T$ must be a vertex in at least one of $T_0$ or $T_1$. In other words,
	\[
	\{v_0,v_1,v_2,v_3\}\subset \{v_i,v_j,p_1,p_2,p_3\}
	\]
	Moreover, $T_0$ and $T_1$ are distinct tetrahedra, meaning they cannot both contain all 4 vertices of $T$. As such, there exists 2 vertices of $T$, without loss of generality call them $v_0$ and $v_1$, such that 
	\[
	v_0\in T_0 \setminus T_1 \quad \text{and} \quad v_1\in T_1 \setminus T_0
	\]
	Since $T_0$ and $T_1$ only differ by one vertex, it follows that $v_i = v_0$ and $v_j = v_1$.
	This then implies that,
	\[
	\{v_2,v_3\}\subset \{p_1,p_2,p_3\}
	\]
	After possibly relabeling the indices we can assume $p_2 = v_2$ and $p_3 = v_3$. Therefore,
	\[
	T_0=\operatorname{conv}\{v_0,v_2,v_3,p_1\} \quad \text{ and } \quad T_1 = \operatorname{conv}\{v_1,v_2,v_3,p_1\}
	\]
	Because $v_2 = p_2 \in \Pi$ and $v_3=p_3 \in \Pi$, the segment $[v_2,v_3]=[p_2,p_3] \subset \Pi$ is contained within $\Pi$. Thus, the edge $[v_2,v_3] \in E(T)$ is contained within $\Pi$. Next, observe that $p_1$ cannot be a vertex in $T$, otherwise $\Pi$ would contain the face of $T$ given by $\operatorname{conv}\{p_1,v_2,v_3\}$ contradicting the assumption that $\Pi$ induces a non-degenerate subdivision. Hence, $p_1$ is in the interior of some edge $e'\in E(T)$. Since $p_1\in \Pi$, the endpoints of $e'$ lie in different half-spaces. However, the only edge with this property is $[v_0,v_1]$. Therefore, $p_1\in \operatorname{relint}([v_0,v_1])$. Since $[v_2,v_3]$ and $[v_0,v_1]$ are opposite edges, this proves the claim.

	$(\Leftarrow)$ Assume the affine plane $\Pi$ contains an edge $e \in E(T)$ and intersects the opposite edge $e' \in E(T)$ at exactly one point $p \in \operatorname{relint}(e')$. Without loss of generality, let
	\begin{equation*}
		e = [v_2,v_3], \quad e' = [v_0,v_1]
	\end{equation*}
	Then
	\begin{equation*}
		\Pi = \operatorname{aff}\{v_2,v_3,p\} \quad \text{ and } \quad T\cap \Pi = \operatorname{conv}\{v_2,v_3,p\}
	\end{equation*}
	Since $p \in \operatorname{relint}([v_0,v_1])$, the plane $\Pi$ separates $v_0$ and $v_1$. Thus, we may choose $H^-$ and $H^+$ to be the closed half-spaces determined by $\Pi$, such that $v_0 \in H^-$ and $v_1 \in H^+$. 
	
	Then let us define 
	\begin{equation*}
		P_0 = T \cap H^-, \quad P_1 = T \cap H^+
	\end{equation*}
	Since $P_0 = T \cap H^-$ is the intersection of a tetrahedron with a half-space, its vertices are among the vertices of $T$ lying in $H^-$, together with points of $\Pi$ that intersect edges of $T$ joining opposite half-spaces. Because $v_0,v_2,v_3 \in H^-$, $v_1, v_2,v_3 \in H^+$, and $\Pi$ intersects the edge $[v_0,v_1]$ at the point $p$, it follows that the only vertices of $P_0$ are $v_0,v_2,v_3,$ and $p$. Similarly, the only vertices of $P_1$ are $v_1,v_2,v_3,$ and $p$. Therefore,
	\begin{equation*}
		P_0 = \operatorname{conv}\{v_0,v_2,v_3,p\}, \quad P_1 = \operatorname{conv}\{v_1,v_2,v_3,p\}
	\end{equation*}
	Furthermore, $p \not \in \operatorname{aff}\{v_0,v_2,v_3\}$ and $p \notin \operatorname{aff}\{v_1,v_2,v_3\}$. Both vertex sets consist of four affinely independent  points, thus, both $P_0$ and $P_1$ are tetrahedra with 
	\begin{equation*}
		T = P_0 \cup P_1
	\end{equation*}
	Therefore, $\Pi$ induces a tetrahedral subdivision of $T$.
\end{proof}

\section{Properties of Parametrized Subdivisions and Angle Functions}\label{anglefunction}
In the previous section, we provided a complete characterization of affine planes that induce tetrahedral subdivisions. There, the intersection point $p$ was treated as a fixed object, yielding a single subdivision of $T$. In this section, we transition to a continuous framework by allowing $p$ to vary as an independent parameter. This enables us to describe a $p$-dependent family of affine planes and their corresponding tetrahedral subdivisions, along with the dihedral angle functions used to quantify the resulting geometric configurations.

\subsection{Dihedral Angle functions for fixed Subdivisions}
Prior to introducing parametrized subdivisions, we ground ourselves in the discrete setting of fixed subdivisions with a fairly well-known angle function.
\begin{defin}
	Suppose $T\subset \mathbb{R}^3$ is a tetrahedron and let $e \in E(T)$. Let $F_1,F_2 \in F(T)$ be the two faces incident to $e$, and for each face $F_k$, let $\hat{n}_k$ be the inward facing unit normal vector to $F_k$. The \textit{dihedral angle} function with respect to $T$ is defined by 
	\begin{equation}\label{FIXED}
		\alpha_T:E(T) \to [0,\pi], \quad  \alpha_T(e) := \arccos(\langle \hat{n}_1,\hat{n}_2\rangle )
	\end{equation}
\end{defin}
\begin{remark}
	The unit normal $\hat{n}_k$ may be viewed equivalently as the unit normal to the affine plane containing the face $F_k$, since each face of a tetrahedron lies in a unique affine plane. Hence, measuring the angle between faces is equivalent to measuring the angle between their supporting planes.
\end{remark}
Intuitively, $\alpha_T$ measures the dihedral angle between the two planes containing the faces of $T$ incident to an edge. This interpretation is reflected in the second argument of Equation \ref{FIXED}, which follows directly from
\begin{equation}
	\cos(\alpha_T(e)) = \langle \hat{n}_1, \hat{n}_2 \rangle,
\end{equation}
where $\langle \cdot, \cdot \rangle$ denotes the standard inner product. By construction, $\alpha_T$ is a bounded function on the discrete set $E(T)$, and is therefore not continuous in any meaningful topological sense (we assume $\mathbb{R}^3$ to be equipped with the standard Euclidean metric). In the context of general tetrahedra, this function may seem rather mundane, but relative to tetrahedral subdivisions, they become much more meaningful. Indeed, if $\Pi$ is a fixed affine plane that induces a non-degenerate tetrahedral subdivision $T= T_0 \cup T_1$, then $\alpha_T$ may be decomposed into two separate functions:
\begin{equation}\label{A12}
	\alpha_{T_0}:E(T_0) \to [0,\pi], \quad \alpha_{T_1}:E(T_1) \to [0,\pi]
\end{equation}
These functions are sub-tetrahedra-specific dihedral angle functions defined on their respective edge sets. The behavior of these functions for distinct edges of $T_0$ and $T_1$ not contained in $E(T_0)\cap E(T_1)$ are fairly trivial and analogous to $\alpha_T$. The more interesting behavior occurs on the shared face $T \cap \Pi$. Let 
\begin{equation*}
	E ( T \cap \Pi) = \{e,e_1,e_2\}
\end{equation*}
where $e \in E(T)$ is the original edge of $T$ contained in $\Pi$, and $e_1,e_2 \notin E(T)$ are the two edges introduced by the subdivision of $T$. Since $e,e_1,e_2 \in E(T_0) \cap E(T_1)$, the corresponding angle functions $\alpha_{T_0}$ and $\alpha_{T_1}$ satisfy the following algebraic relations:
\begin{enumerate}
	\item For the edge $e \in E(T)$ contained in $\Pi$,
	\begin{equation}\label{FIXEDFUNC1}
		\alpha_{T_0}(e)+\alpha_{T_1}(e) = \alpha_T(e)
	\end{equation}
	\item For the remaining edges $e_1,e_2$ formed by the subdivision,
	\begin{equation}\label{FIXEDFUNC2}
		\alpha_{T_0}(e_1)+\alpha_{T_1}(e_1) = \pi, \quad \alpha_{T_0}(e_2)+\alpha_{T_1}(e_2) = \pi 
	\end{equation}
\end{enumerate}

The first relationship described by Equation \ref{FIXEDFUNC1} is a direct result of the angle being split by the subdivision along $e$ as depicted in Figure \ref{ANGLE1}. The second relationship from Equation \ref{FIXEDFUNC2} holds because $e_1$ and $e_2$ are not original edges of $T$, but instead lie within faces $F_1, F_2 \in F(T)$, respectively. Consequently, the dihedral angle at each edge is measured with respect to the plane containing $F_1$ and $F_2$, respectively. 

\begin{figure}
	\centering
	\includegraphics[scale=.88]{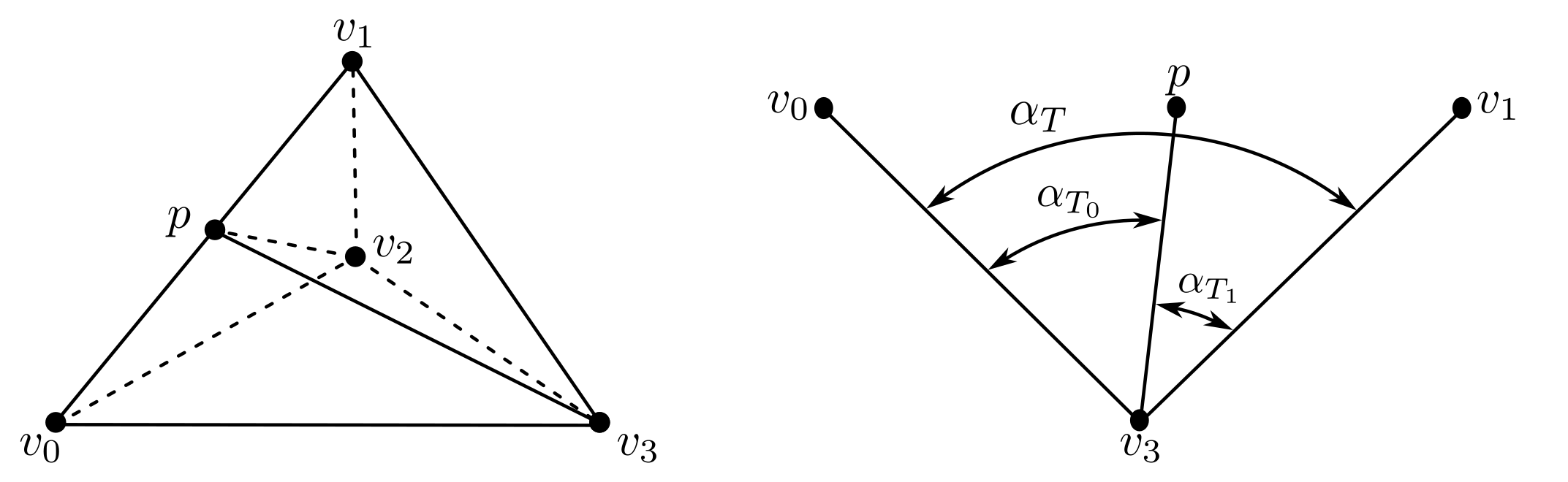}
	\caption{Relationship shared by $\alpha_{T_0}$ and $\alpha_{T_1}$ described in Equation \ref{FIXEDFUNC1}, where the dihedral angle of original edge $e_0 =[v_2,v_3]$ is split by a generic tetrahedral subdivision. (\textit{Left}) Birds-eye view of $T$ illustrating subdivision and the shared face. (\textit{Right}) 2 dimensional side view relative to the vertex $v_3$, and the associated angles measured between each face.}
	\label{ANGLE1}
\end{figure}

\subsection{Parametrized Subdivisions and Configuration Spaces}
As described in Lemma \ref{TS}, any tetrahedral subdivision of $T$ is induced by an affine plane $\Pi$ that contains an edge $e \in E(T)$ and intersects the edge $e'$ opposite to $e$ at exactly one point $p \in \operatorname{relint}(e')$. To produce a family of tetrahedral subdivisions on $T$, we allow $p$ to vary along $e'$.
\begin{defin}\label{FAMS}
	Let $e,e' \in E(T)$ be opposite edges and let $p \in e'$. Then $\left\{\Pi_p\right\}_{p\in e'}$ is a \textit{one-parameter family of affine planes} with
	\begin{equation}\label{3.6}
		\Pi_p:= \operatorname{aff}\{e \cup \{p\}\}
	\end{equation}
	that induces the associated family of subdivisions $\left\{T_0^p, T_1^p\right\}_{p \in e'}$ where $T_0^p$ and $T_1^p$ are uniquely determined by each $\Pi_p$ and satisfy
	\begin{equation}
		T = T_0^p\cup T_1^p
	\end{equation}
	for all $p \in e'$.
\end{defin}

\begin{figure}[h!]
	\centering
	\includegraphics[scale=.95]{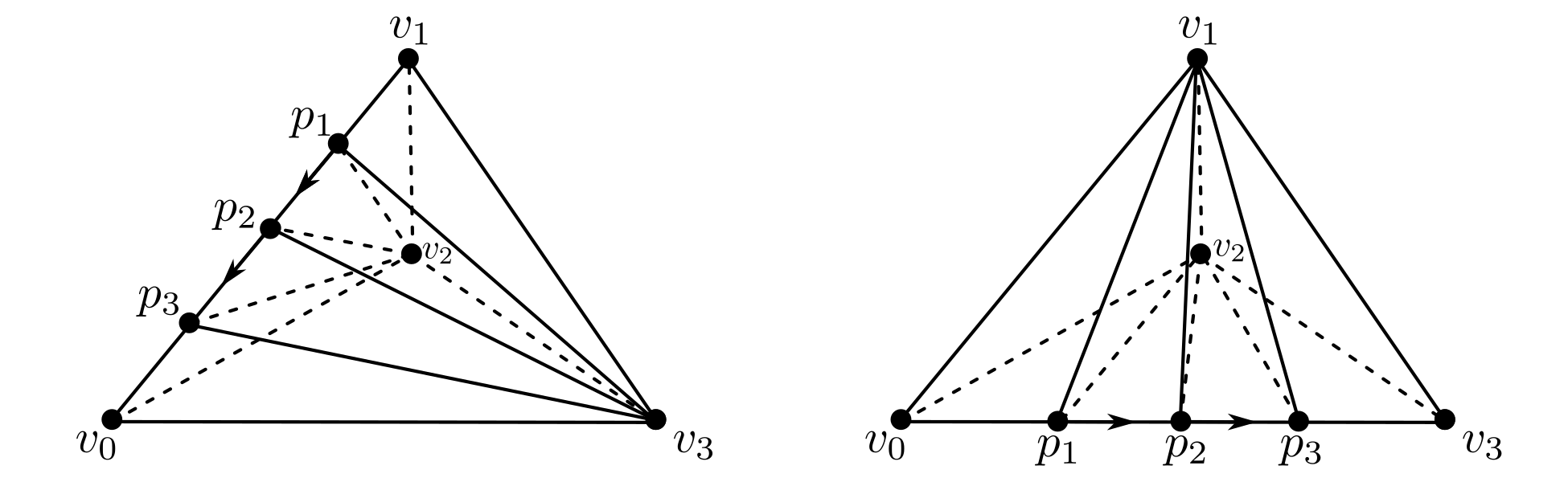}
	\caption{Two unique parametrized families of subdivisions on the same tetrahedron $T$ with freeze frames at $p_1,p_2,p_3\in [v_0,v_1]$ and $p_1,p_2,p_3 \in [v_0,v_3]$ as $p$ varies from $v_1$ to $v_0$ and $v_0$ to $v_3$ respectively.}
	\label{FAMILYSUB}
\end{figure}

By construction, each affine plane $\Pi_p \in \{\Pi_p\}_{p\in e'}$ contains the fixed edge $e$ and uniquely determines two $p-$dependent edges not belonging to $E(T)$, which we collectively denote by $e(p)$. Unlike Lemma \ref{TS}, Definition \ref{FAMS} places no interior restriction on the interval $e'$ for which $p$ varies in. As a result, the family $\{\Pi_p\}_{p \in e'}$ contains two distinct affine planes defined at the vertex endpoints of $e'$, each corresponding to a degeneration of $T_0^p$ or $T_1^p$ respectively. In each case, the nondegenerate member of the induced bisection coincides with the original tetrahedron $T$, thereby necessitating the inclusion of $T$ in the family $\{T_0^p,T_1^p\}_{p \in e'}$. Naturally, these endpoint configurations will serve as limiting objects to determine the behavior of one sided limits of functions arising from the family $\{T_0^p,T_1^p\}_{p \in e'}$ as $p$ varies over $e'$.

To begin to analyze the underlying structure of these families, we introduce a 3-dimensional configuration space that encodes all relevant geometric data (namely edges, subdivision index, and parameter position) in a single indexing space.

\begin{defin}
	Let $T$ be a tetrahedron and $e,e' \in E(T)$. If $\{\Pi_p\}_{p \in e'}$ is a one-parameter family of affine planes that induces the family of tetrahedral subdivisions $\{T_0^p,T_1^p\}_{p \in e'}$ of $T$, then the configuration space relative to $\{T_0^p,T_1^p\}_{p \in e'}$ is
	\begin{equation}
		\xi = \{(e,i,p) \mid p \in \operatorname{relint} (e'), i \in \{0,1\}, e \in E(T_i^p)\}
	\end{equation}
\end{defin}
The configuration space $\xi$ captures every possible combinatorial-geometric configuration of $T$ under a single parametrized family of subdivisions. Naturally, this configuration space decomposes nicely according to the two sub-tetrahedra $T_0^p$ and $T_1^p$. Accordingly, we define the subsets
\begin{equation}
	\xi_0 := \{(e,i,p)\in \xi \mid i=0\},
	\qquad
	\xi_1 := \{(e,i,p)\in \xi \mid i=1\}.
\end{equation}
which capture the combinatorial-geometric configurations of $T_0^p$ and $T_1^p$ individually. These subsets allow us to separate quantities defined on each sub-tetrahedron, enabling independent analysis of $T_0^p$ and $T_1^p$. It is important to note, though, that $\xi$, $\xi_0$ and $\xi_1$ lack any canonical topology, metric, or smooth structure, and thus, are all purely set-theoretic objects. Consequently, any such function defined on the domain of $\xi$ does not support an intrinsic analytical structure. In particular, since $\xi$ lacks a topological structure, notions such as continuity or monotonicity are not defined for functions on $\xi$.

\subsection{The Dihedral Angle Function on Parametrized Subdivisions}
Thus far, we have studied the dihedral angle function for fixed subdivisions and constructed a parametrized family of subdivisions on $T$ together with their associated configuration spaces. We now introduce a generalized dihedral angle function defined on this family of subdivisions.

\begin{defin}
	Let $T$ be a tetrahedron and let $\{\Pi_p\}_{p \in e'}$ be a one-parameter family of affine planes inducing a family of tetrahedral subdivisions $\{T_0^p,T_1^p\}_{p \in e'}$ of $T$. Let $\xi$ be the associated configuration space. For each $(e,i,p)\in \xi$, let $F_1$ and $F_2$ denote the two faces of $T_i^p$ incident to $e$, and let $\hat n_1(p)$ and $\hat n_2(p)$ denote their corresponding inward facing unit normals. The associated dihedral angle function is
	\begin{equation}
		\alpha : \xi \to [0,\pi],
		\qquad
		\alpha(e,i,p)
		:= \arccos\!\big(\langle \hat n_1(p),\hat n_2(p)\rangle\big).
	\end{equation}
	Define two of its restrictions as
	\begin{equation}
		\alpha_0 := \alpha|_{\xi_0},
		\qquad
		\alpha_1 := \alpha|_{\xi_1}.
	\end{equation}
	When no ambiguity can arise, we suppress the fixed index in the domain and write
	\begin{equation}
		\alpha_0(e,p) := \alpha(e,0,p),
		\qquad
		\alpha_1(e,p) := \alpha(e,1,p).
	\end{equation}
\end{defin}
Much like the configuration space $\xi$, we treat $\alpha$ as a tool to capture geometric measurements associated with configurations of the subdivision family. Although $\alpha$ may seem to represent a continuous measurement over $\{T_0^p,T_1^p\}_{p\in e'}$, at its core, it is still a discrete set-theoretic function, and remains undefined for degenerate subdivisions of $T$.
\begin{figure}[h!]
	\centering
	\includegraphics[scale=.95]{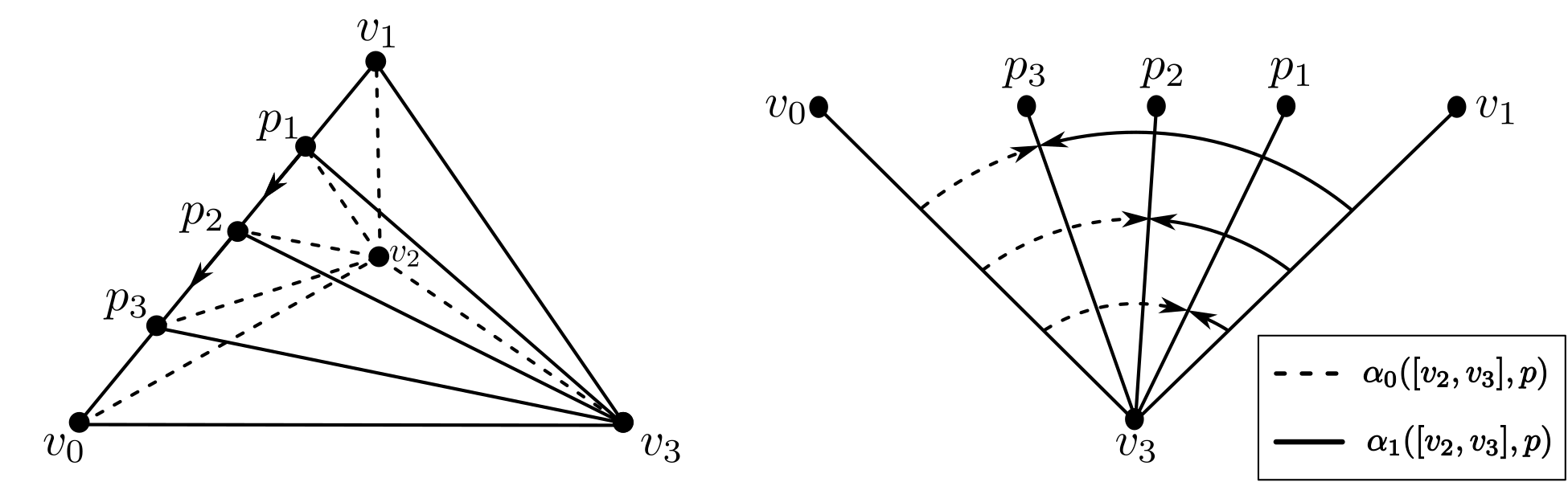}
	\caption{Single Parameter Functions \textit{(Left)} Parametrized family of subdivisions from Figure \ref{FAMILYSUB} \textit{(Right)} Angle functions for the fixed edge $[v_2,v_3]$ providing a continuous measure as $p$ varies, and satisfying Equation \ref{ALG1}}
\end{figure}

The restrictions $\alpha_0$ and $\alpha_1$ serve as counterparts to the angle functions $\alpha_{T_0}$ and $\alpha_{T_1}$ of a fixed subdivision introduced in Equation \ref{A12}. Specifically, they assign to each admissible configuration $(e,i,p)$ the dihedral angle at the edge $e$, measured in the tetrahedron $T_i^p$. As analogs to the fixed angle functions $\alpha_{T_0}$ and $\alpha_{T_1}$, these functions continuously satisfy identical algebraic relations as $p$ varies. Indeed, we find that for the fixed edge $e$ contained in each affine plane of $\{\Pi_p\}_{p \in e'}$, the restrictions of $\alpha$ satisfy
\begin{equation}\label{ALG1}
	\alpha_{0}(e,p) + \alpha_{1}(e,p) = \alpha_T(e) \quad \forall p\in \operatorname{relint}(e')
\end{equation}
and for the edges $e_1(p)$ and $e_2(p)$ that are continuously defined as $p$ varies,
\begin{equation}\label{ALG2}
	\alpha_{0}(e_1(p),p) + \alpha_{1}(e_1(p),p) = \pi, \quad 
	\alpha_{0}(e_2(p),p) + \alpha_{1}(e_2(p),p) = \pi \quad \forall p\in \operatorname{relint}(e')
\end{equation}
However, special attention must be paid to the domain on which these relations are well defined. For the fixed dihedral angle functions, we assumed that the fixed affine plane $\Pi$ induced a nondegenerate subdivision of $T$, ensuring that each function was well defined on its corresponding edge set. In contrast, any family $\{T_0^p,T_1^p\}_{p \in e'}$ has that $T_0^p$ and $T_1^p$ degenerate at opposite endpoints of (e'). Consequently, the associated relations are generally well defined only for $p \in \operatorname{relint}(e')$.

\subsection{Induced Single-Parameter Functions}

As imposed by the configuration space $\xi$, the dihedral angle function $\alpha:\xi \to [0,\pi]$ relative to $\{T_0^p,T_1^p\}_{p\in e'}$ fails to exhibit properties such as continuity and monotonicity, a critical factor in analyzing dynamic systems. To produce a practical realization of $\alpha$, we define auxiliary functions used for analysis as pullbacks along parameter maps $p \mapsto (e,i,p)$. In doing so, we avoid the need to employ an explicit topology on $\xi$. Furthermore, since each edge $e' \subset \mathbb{R}^3$ is homeomorphic to a closed interval in $\mathbb{R}$, we will often identify $p \in e'$ with a real parameter in $[0,1]$. Consequently, the $p$-dependent functions introduced below may be viewed as real-valued functions of a single real variable.

\begin{defin}\label{ANGLEFUNCS}
	Suppose $\{\Pi_p\}_{p \in e'}$ is a one-parameter family inducing the tetrahedral subdivisions $\{T_0^p,T_1^p\}_{p \in e'}$. Let $\xi$ denote the associated configuration space and let $\alpha : \xi \to [0,\pi]$ be the corresponding dihedral angle function, with restrictions $\alpha_i := \alpha|_{\xi_i}$ for $i \in \{0,1\}$. For the fixed edge $e$ contained in each $\Pi_p$, we define the single-parameter function
	\begin{equation}\label{FUNC1}
		p \longmapsto \alpha_i(e,p), \qquad i \in \{0,1\}.
	\end{equation}
	and for $e(p)$, an edge determined by $\Pi_p$, the single-parameter function
	\begin{equation}\label{FUNC2}
		p \longmapsto \alpha_i(e(p),p), \qquad i \in \{0,1\}.
	\end{equation}
\end{defin}
Prior to deriving important properties relating to these functions, we take a moment to appreciate and more explicitly analyze the underlying pullback of $\alpha$ taking place. Although the dihedral angle function $\alpha$ is defined on the configuration space $\xi$, the functions of Definition \ref{ANGLEFUNCS} are obtained by fixing a geometric quantity recorded by $\alpha$ and allowing the parameter $p$ to vary. More precisely, for the fixed edge $e$, the map $p \longmapsto \alpha_i(e,p)$ selects a one-parameter subset of $\xi$, and the corresponding angle function is obtained by the composition
\begin{equation*}
	p \longmapsto (e,i,p) \longmapsto \alpha(e,i,p)
\end{equation*}
Likewise, for the edge $e(p)$ determined by $\Pi_p$, the map $p \longmapsto \alpha_i(e(p),p)$ selects a parametrized family of configurations, yielding the composition
\begin{equation*}
	p \longmapsto (e(p),i,p) \longmapsto \alpha(e(p),i,p)
\end{equation*}
In this way, these functions measure the evolution of specific dihedral angles throughout the family of subdivisions. Consequently, all subsequent analysis is reduced to the study of real-valued functions of the single parameter $p$.

In order to use these newly defined functions to their fullest extent, we demonstrate that they exhibit two key properties.

\begin{lemma}[Continuity of $p$-Functions]\label{PROP1}
	For each $i \in \{0,1\}$, the functions
	\begin{equation*}
		p \longmapsto \alpha_i(e,p),
		\qquad
		p \longmapsto \alpha_i(e(p),p)
	\end{equation*}
	are continuous on $\operatorname{relint}(e')$.
\end{lemma}

\begin{proof}
	Assume $T$ is a tetrahedron with opposite edges $e',e \in E(T)$, and the one parameter family $\{\Pi_p\}_{p \in e'}$ defines a family of subdivisions $\{T_0^p,T_1^p\}_{p \in e'}$ on $T$. The tetrahedra $T_i^p$ are obtained by intersecting fixed affine substructures of $T$ with the moving plane
	\begin{equation*}
		\Pi_p = \operatorname{aff}\{e \cup \{p\}\}
	\end{equation*}
	As a consequence, each vertex of $T_i^p$ arises as the solution of a finite system of affine equations whose coefficients depend smoothly on $p \in \operatorname{relint}(e')$. Thus, all vertices of $T_i^p$ vary continuously with respect to $p$.
	
	It follows then that each edge vector of $T_i^p$ is of the form
	\begin{equation*}
		u(p) = v_k(p)-v_l(p)
	\end{equation*}
	where $v_k(p)$ and $v_l(p)$ are vertices of $T_i^p$. Since each vertex depends continuously on $p$ in $\operatorname{relint}(e')$, every such vector $u(p)$ is a continuous $\mathbb{R}^3$-valued function of $p$. Now, since any face of $T_i^p$ is spanned by two such edge vectors, call them $u_1(p)$ and $u_2(p)$, the associated normal vector can be written as
	\begin{equation*}
		n(p) = u_1(p) \times u_2(p)
	\end{equation*}
	Because the cross product of two vectors is bilinear, $n(p)$ is a continuous function of $p$. Furthermore, since no face degenerates on $\operatorname{relint}(e')$, we have that $n(p) \neq 0$ for all $p \in \operatorname{relint}(e')$. Thus, the normalized vector
	\begin{equation*}
		\hat{n}(p) = \frac{n(p)}{||n(p)||}
	\end{equation*}
	is well defined and continuous on $\operatorname{relint}(e')$. Consequently, for both fixed edges $e$ and moving edges $e_j(p)$, the two incident face unit normals $\hat{n}_1(p)$ and $\hat{n}_2(p)$ depend continuously on $p$. Thus, the map
	\begin{equation*}
		p \mapsto \langle \hat{n}_1(p),\hat{n}_2(p) \rangle
	\end{equation*}
	is continuous and has an image contained within the interval $(-1,1)$. Finally, since 
	\begin{equation*}
		\arccos:(-1,1) \to [0,\pi]
	\end{equation*}
	is continuous, it follows that the functions
	\begin{equation*}
		p \longmapsto  \alpha_i(e,p), \quad p \longmapsto  \alpha_i(e(p),p) 
	\end{equation*}
	are continuous on $\operatorname{relint}(e')$.
\end{proof}

\begin{lemma}
	For each $i \in \{0,1\}$, the functions 
	\begin{equation*}
		p \longmapsto \alpha_i(e,p), \quad p \longmapsto \alpha_i(e(p),p)
	\end{equation*}
	are monotonic on $\operatorname{relint}(e')$.
\end{lemma}
\begin{proof}
	Without loss of generality, assume $T = \operatorname{conv}\{v_0,v_1,v_2,v_3\}$, $e' = [v_0,v_1]$ and define the one parameter family of affine planes $\{\Pi_p\}_{p \in [v_0,v_1]}$ such that each plane $\Pi_p$ fully contains the edge $[v_2,v_3]$. Then the corresponding family of subdivisions $\{T_0^p,T_1^p\}_{p \in [v_0,v_1]}$ contains elements of the form
	\begin{equation*}
		T_0^p = \operatorname{conv} \{v_0,v_2,v_3,p\}, \quad T_1^p = \operatorname{conv} \{v_1,v_2,v_3,p\}
	\end{equation*}
	that have shared faces defined by
	\begin{equation*}
		T_0^p \cap T_1^p = \operatorname{conv}\{v_2,v_3,p\}    
	\end{equation*}
	and has a configuration space $\xi$ that defines the general dihedral angle function $\alpha: \xi \to [0,\pi]$. We begin by showing that the function 
	\begin{equation*}
		p \longmapsto \alpha_i([v_2,v_3],p)
	\end{equation*}
	associated with the fixed edge $[v_2,v_3]$ is increasing for $i=0$, and likewise decreasing for $i=1$. Fix $r,q\in \operatorname{relint}([v_0,v_1])$ with $r<q$. Since $r \in \operatorname{relint}([v_0,q])$, we have that $T_0^r \subset T_0^q$. Moreover, $r$ induces a tetrahedral subdivision of $T_0^q$ into the sub-tetrahedra
	\begin{equation*}
		T_0^r= (T_0^q)^r_0 = \operatorname{conv}\{v_0,v_2,v_3,r\}, \quad (T_0^q)^r_1 = \operatorname{conv}\{q,v_2,v_3,r\}
	\end{equation*}
	Let
	\begin{equation*}
		\alpha_{(T_0^q)^r_0}: E((T_0^q)^r_0) \to [0,\pi], \quad \alpha_{(T_0^q)^r_1}: E((T_0^q)^r_1) \to [0,\pi]
	\end{equation*}
	denote the fixed dihedral angle functions associated with these tetrahedra. Since $(T_0^q)^r_0$ and $(T_0^q)^r_1$ form a tetrahedral subdivision of $T_0^q$ along the edge $[v_2,v_3]$, we have that
	\begin{equation*}
		\alpha_0([v_2,v_3],q) = \alpha_{(T_0^q)^r_0}([v_2,v_3])+\alpha_{(T_0^q)^r_1}([v_2.v_3])
	\end{equation*}
	Since both $(T_0^q)^r_0$ and $(T_0^q)^r_1$ are non-degenerate, this sum is positive and well defined. Then it follows that
	\begin{align*}
		\alpha_0([v_2,v_3],r) = \alpha_{(T_0^q)^r_0}([v_2,v_3]) &< \alpha_{(T_0^q)^r_0}([v_2,v_3])+\alpha_{(T_0^q)^r_1}([v_2.v_3]) \\
		&= \alpha_{0}([v_2,v_3],q)
	\end{align*}
	Thus $\alpha_i([v_2,v_3],p)$ is increasing for $i=0$, and likewise, decreasing for $i=1$. 
	
	Because of our assumptions on $e$ and $e'$ the edge $e(p)\in E(T_i^p)$, which depends on $p$, must be of the form $[v_2,p]$ or $[v_3,p]$. To prove that $\alpha_i(e(p),p)$ is monotonic we will assume that $e(p) = [v_2,p]$ and that $i=0$, the other cases are identical.
	
	To show $\alpha_0([v_2,p],p)$ is monotonic it suffices to show that $\alpha_0([v_2,p],p)$ is either a constant function or injective. Indeed, if $\alpha_0([v_2,p],p)$ is injective it is also continuous by Lemma 3.6, it follows that $\alpha_0([v_2,p],p)$ must be monotonic since all injective continuous functions on an interval are monotonic. Now, we aim to show that if $\alpha_0([v_2,p],p)$ is not injective it must be constant. 
	
	Suppose $g$ is not injective, then there exists $p,q\in \operatorname{relint}([v_0,v_1])$ with $p\neq q$ and $\alpha_0([v_2,p],p)=\alpha_0([v_2,q],q)$. Observe that applying isometries to $T$ does not change the value of its dihedral angles, so the value of $\alpha_0([v_2,p],p)$ remains invariant under translations and rotations of $T$. As such, we will translate $T$ so that $v_2 = 0 \in \mathbb{R}^3$ and rotate $T$ so the face $\operatorname{conv}\{v_0,v_1,v_2\}$ lies within the $xy$-plane and reflect about $xy$-plane so that $v_3$ has positive $z$ coordinate. This last reflection ensures that the inward facing normal of the face $\operatorname{conv}\{v_0,v_1,v_2\}$ is $(0,0,1)$. With this configuration we know that,
	$\alpha_0([v_2,p],p)$ is the angle between the inward facing normals of the planes $\Pi_p$ and $\operatorname{aff}\{v_0,v_2, p\}$
	However,
	\[
	\operatorname{conv}\{v_0,v_2,p\} \subset \operatorname{conv}\{v_0,v_1,v_2\} \implies \operatorname{aff}\{v_0,v_2, p\} = \operatorname{aff}\{v_0,v_1, v_2\}
	\]
	By the configuration of $T$, the last plane is simply the $xy$-plane, hence $\alpha_0([v_2,p],p)$ is the angle between the inward facing normal of $\Pi_p$ and $(0,0,1)$. Identically, $\alpha_0([v_2,q],q)$ is the angle between the inward facing normal of $\Pi_q$ and $(0,0,1)$. Because $\alpha_0([v_2,p],p)=\alpha_0([v_2,q],q)$, we know that $\Pi_p$ and $\Pi_q$ form the same dihedral angle with the $xy$-plane.
	
	Next, let $R$ be an orientation-preserving rotation about the $z$-axis such that $R(q) \in \operatorname{span}\{p\}$.
	This means that the two planes $R(\operatorname{span}\{q,v_3\})$ and $\operatorname{span}\{p,v_3\}$ intersect the $xy$-plane along the same line, $\operatorname{span}\{p\}$. Since $R$ is an orientation preserving isometry, it follows that $R(\operatorname{span}\{q,v_3\})$ and $\operatorname{span}\{p,v_3\}$ form the same angle with the $xy$-plane, moreover, the normal vectors of these two planes point in the same direction relative to their line of intersection. All of this implies that these planes are equal,
	\[
	R(\operatorname{span}\{q,v_3\}) = \operatorname{span}\{p,v_3\}
	\]
	
	However, this means that the axis of rotation, the $z$-axis, is the line of intersection between the two planes $\Pi_p$ and $\Pi_q$.
	\[
	\operatorname{span}\{p,v_3\}\cap \operatorname{span}\{q,v_3\} = \operatorname{span}\{v_3\}
	\] 
	In particular, $v_3$ lies within the $z$-axis. Thus, for any $r\in \operatorname{relint}([v_0,v_1])$ we have $\Pi_r=\operatorname{aff}\{v_2,v_3,r\}=\operatorname{span}\{v_3,r\}$ and it is clear that $\Pi_r$ forms a right angle with the $xy$-plane. Hence, the inward facing unit normal vector of $\Pi_r$, call it $\hat n_r$, has no $z$-component. Meaning,
	\[
	\langle n_r ,(0,0,1) \rangle = 0
	\]
	This lets us calculate $\alpha_0([v_2,r],r)$,
	\[
	\alpha_0([v_2,r],r) =\arccos\!\big(\langle \hat n_r, (0,0,1)\rangle\big) = \arccos(0) = \frac{\pi}{2}
	\]
	This is true for any $r\in \operatorname{relint}([v_0,v_1])$, so $p \longmapsto \alpha_0([v_2,p],p)$ is constant.

\end{proof}

\section{Classifications and The Dihedral Criterion}\label{tetrahedra}
Over 2300 years ago, Euclid of Alexandria, the father of Euclidean Geometry, classified triangles in $\mathbb{R}^2$ by the lengths of their edges and the interior angles measured between their edges. In particular, Euclid formalized the terms "obtuse," "acute," and "right" to describe both the types of angles that may occur and the subsequent types of triangles. By formally defining these terms, Euclid provided mathematicians with the tools necessary to, when possible, identify the underlying simplified structure of more complex geometric problems.

For tetrahedra in $\mathbb{R}^3$ analyzing underlying structure becomes much more difficult, as the number of empirical quantities determining particular properties of a tetrahedra vastly increases. To avoid an overcomplication of scope, we propose a dihedral angle classification that aims to collapse the infinite number of tetrahedra that may occur, while preserving a level of freedom in argument.

\begin{defin}
	A tetrahedron $T$ is said to be \textit{Class $n$} if it has exactly $n$ obtuse dihedral angles. A \emph{subclass} of Class $n$ tetrahedra is an equivalence class of 
	configurations of $n$ obtuse edges under relabeling of vertices.
\end{defin}

While not geometrically exact representations, for the purpose of visual simplicity, we switch to using the full graph $K_4$ for illustrating unique subclasses. Subclasses are distinguished by letter labels according to the adjacency structure of these obtuse edges. Configurations in which all obtuse dihedral angles are incident to a single vertex are labeled (a), while configurations not sharing a common vertex are labeled (b). Since Class 1 tetrahedra admit only one subclass, no subclass identifier is assigned. 
\begin{figure}[h!]
	\centering
	\includegraphics[scale=.8]{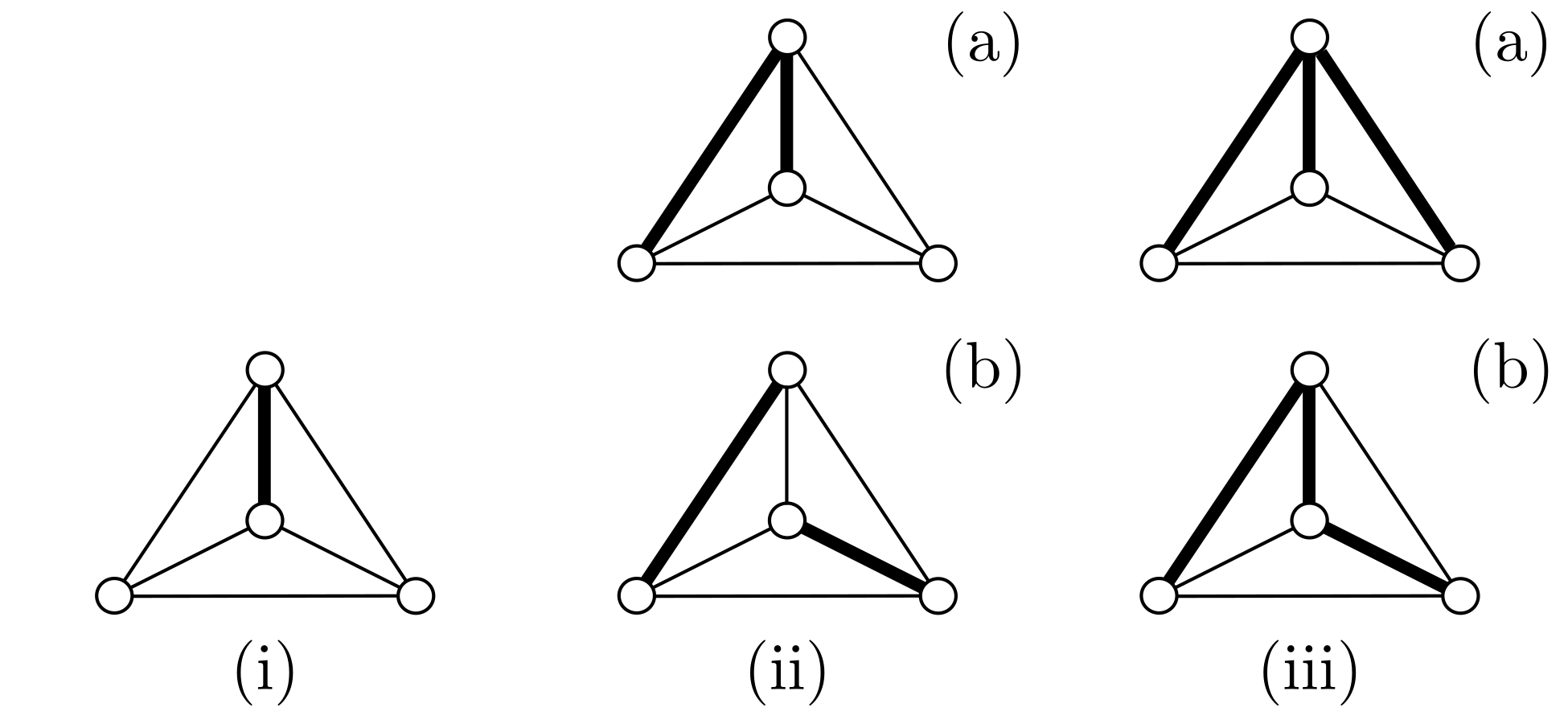}
	\caption{Various Subclasses of Class $n$ tetrahedra. Thickened edges indicate where obtuse dihedral angles occur. (i) Class 1, (ii) Class 2 subclasses, (iii) Class 3 subclasses. }
	\label{SUB}
\end{figure}

From Figure \ref{SUB}, two  immediate observations follow; Class 3 does not contain a subclass corresponding to a 3-Cycle, and for $n >3$ no such Class $n$ tetrahedron exists. The former can be verified by expressing the supposed face $F$ containing three obtuse dihedral angles as the intersection of three half-spaces. In this case, the orthogonal projection of the opposite vertex would necessarily lie outside all three associated half-planes simultaneously, which is impossible. The latter follows from a theorem resulting from an open problem initially proposed by Klamkin and Pook \cite{KlamkinPook1988}, and later proven by Leng \cite{Leng2003}.

\begin{thm}\label{T1}
	There exists at least $n$ acute dihedral angles in any $n$-simplex, and there exists an $n$-simplex which has only $n$ acute dihedral angles.
\end{thm}

The quantification of the minimal number of acute dihedral angles in a tetrahedron imposes a corresponding upper bound on the number of obtuse dihedral angles, thereby allowing the use of counting arguments in place of explicit geometric proofs. With this established, we now move to discuss a subtle, yet important property pertaining to the projective nature of height function tetrahedra. In particular, we introduce a tool for determining the existence of a foundation face $F$ of $T$.

\begin{lemma}[Height-Dihedral Criterion]\label{L1}
	Let $T  \subset \mathbb{R}^3$ be a tetrahedron and $F$ be a face of $T$ with edges $e_1,e_2,e_3$.
	Then $T$ is a height function relative to the foundation $F$ if and only if the dihedral angles along the edges of $F$ are all not obtuse.
\end{lemma}

\begin{proof}
	$(\Rightarrow)$ Let $T$ be a height function relative to the foundation $F$. After possibly relabeling the vertices of $T$ we may assume $F = \operatorname{conv}\{v_0,v_1,v_2\}$. Notice that applying isometries and scalings to $\mathbb{R}^3$ does not alter the dihedral angles of $T$, as such, we will translate, rotate, reflect, and scale $T$ in order to produce a simpler environment for studying $\alpha_T([v_0,v_1]), \alpha_T([v_0,v_2])$ and $\alpha_T([v_1,v_2])$. After applying a translation, assume $v_0 = 0 \in \mathbb{R}^3$. Next, we apply a rigid rotation of $\mathbb{R}^3$ such that the affine hull of $F$ coincides with the $xy$-plane and the edge $[v_0,v_1]$ lies along the positive $x$-axis. In particular, this means that $v_1,v_2 \in \mathbb{R}^3$ have zero $z$ component. Hence, we may write
	\begin{equation*}
		v_1 = (x_1,0,0), \quad v_2 = (x_2,y_2,0)
	\end{equation*}
	with $x_1>0$. If necessary, reflect $T$ across the $xz$ plane to ensure $y_2>0$, and scale $T$ so that
	\begin{equation*}
		v_3 = (x_3,y_3,1)\quad
	\end{equation*}
	
	Considering the face $K = \operatorname{conv}\{v_0,v_1,v_3\}$ we move to demonstrate that $\theta_1$, the dihedral angle that occurs at the edge $[v_0,v_1]$ shared by $F$ and $K$ must be non-obtuse. Let $P_1 = \operatorname{aff}\{v_0,v_1,v_2\}$ be the plane containing $F$, and $P_2 = \operatorname{aff}\{v_0,v_1,v_3\}$ be the plane containing $K$. Since $v_0=0$ these planes simplify as,
	\[
	P_1= \operatorname{span} \{v_1,v_2\} \quad \text{ and } \quad P_2= \operatorname{span} \{v_1,v_3\}
	\]
	Because $v_1,v_2$ are linearly independent vectors in the $xy$ plane we conclude that $P_1$ is the $xy$ plane, as such, the vector $n=(0,0,1)$ is normal to $P_1$. Also, 
	\begin{equation*}
		m = v_1\times v_3 = (x_1,0,0) \times (x_3,y_3,1) = (0,-x_1,x_1y_3)
	\end{equation*}
	is normal to $P_2$. Then the dot product between both normal vectors is 
	\begin{align*}
		\langle m , n \rangle  &= (0,-x_1,x_1y_3) \cdot (0,0,1) \\
		&= x_1y_3
	\end{align*}
	Because $x_1>0$, it remains to show that $y_3 \ge 0$. Since $F$ is a foundation of $T$,
	\begin{equation*}
		\pi_F(v_3)= (x_3,y_3,0) \in F
	\end{equation*}
	Then by the convexity of $F$, there exists $r,s,t \ge 0$ with $r+s+t = 1$ such that
	\begin{equation*}
		(x_3,y_3,0) = rv_0+sv_1+tv_2 = sv_1+tv_2
	\end{equation*}
	Hence 
	\begin{equation*}
		(x_3,y_3,0) = (sx_1+tx_2,ty_2,0) \implies y_3 =ty_2
	\end{equation*}
	Since $y_2 >0$ and $t\ge 0$, we get that $y_3 \ge 0$, and therefore
	\begin{equation*}
		\langle n , m \rangle = x_1y_3 \ge 0 
	\end{equation*}
	Due to the dot product angle formula, we have that $\cos(\alpha_T(e_1)) \ge 0 $, and subsequently $\alpha_T([v_0,v_1]) \le \frac{\pi}{2}$. Because the initial normalization and coordinate transformations may be applied relative to any edge of $F$, the preceding argument applies identically to the remaining edges of $F$. Therefore, the remaining angles $\alpha_T([v_0,v_2])$ and $ \alpha_T([v_1,v_2])$ are non-obtuse.

	$(\Leftarrow)$ Suppose the dihedral angles of the edges of $F=\operatorname{conv}\{v_0,v_1,v_2\}$, given by, 
	\[
	\alpha_T([v_0,v_1]),\alpha_T([v_0,v_2]),\alpha_T([v_1,v_2])
	\]
	are all non-obtuse. To show $F$ is a foundation of $T$ we must show that $\pi_F(x) \in F$ for all $x\in T$. In fact, it is sufficient to show that $\pi_F(v_3)\in F$. Indeed, if $\pi_F(v_3) \in F = \operatorname{conv}\{v_0,v_1,v_2\}$ then, 
	\begin{equation*}
		\pi_F(v_3) = \sigma_0v_0+\sigma_1v_1+\sigma_2v_2 ,  \quad \text{ with } \sigma_i \ge 0 , \text{ and } \sum^2_{i=0} \sigma_i = 1
	\end{equation*}
	By the convexity of $T$, we may express any $x\in T$ as a convex combination:
	\begin{equation*}
		x = \lambda_0v_0+\lambda_1v_1+\lambda_2v_2+\lambda_3v_3 ,  \quad \text{ with } \lambda_i \ge 0 , \text{ and } \sum^3_{i=0} \lambda_i = 1
	\end{equation*}
	Then the orthogonal projection of $x$ onto the plane $\operatorname{aff}(F)$ is given by,
	\begin{align*}
		\pi_F(x) &= \lambda_0\pi_F(v_0)+\lambda_1\pi_F(v_1) +\lambda_2\pi_F(v_2)+\lambda_3 \pi_F(v_3) \\
		&= \lambda_0v_0+\lambda_1v_1+\lambda_2v_2 +\lambda_3 \pi_F(v_3)\\
		&= (\lambda_0 + \lambda_3\sigma_0 )v_0+(\lambda_1+\lambda_3\sigma_1)v_1+(\lambda_2 + \lambda_3\sigma_2) v_2 
	\end{align*}
	Observe that,
	\[
	\lambda_i+\lambda_3\sigma_i \ge 0 \quad \text{ and } \quad \sum_{i=0}^2(\lambda_i+\lambda_3\sigma_i) = \sum_{i=0}^3\lambda_i=1
	\]
	Therefore $\pi_F(x)\in \operatorname{conv}\{v_0,v_1,v_2\} = F$.
	
	With this in mind, we aim to show that $\pi_F(v_3) \in F$. Note that applying isometries and scaling $\mathbb{R}^3$ does not change whether or not $F\subset T$ is a foundation. As such, we will apply the exact same translation, rotations, reflections and scaling to $T$ as we did above.
	
	We start by extending the edge $[v_0,v_1]$ into a line $L=\mathrm{span}\{v_1\}$ (since $v_0 = 0$ and $v_1=(x_1,0,0)$, $L$ is simply the  $x$ axis). Observe that $L$ partitions the plane containing $F$ (the $xy$ plane) into two half spaces, namely
	\begin{equation*}
		H_1^+=\{(x,y):y\geq0\} \quad \text{and} \quad  H_1^-=\{(x,y):y<0\}
	\end{equation*}
	By our configuration of $T$, $v_2 = (x_2,y_2,0)$ with $y_2>0$. Hence, $v_2\in H_1^+$ and $F=\operatorname{conv}\{v_0,v_1,v_2\}\subset H_1^+$. 
	
	Note that the dihedral angle $\alpha_T([v_0,v_1])$ is given by the angle between the planes $\operatorname{aff}\{v_0,v_1,v_2\}$ and $\operatorname{aff}\{v_0,v_1,v_3\}$. Another way to express $\alpha_T([v_0,v_1])$ is as the angle between two vectors $n\in \operatorname{aff}\{v_0,v_1,v_2\}, m\in \operatorname{aff}\{v_0,v_1,v_3\}$ that are both orthogonal to the intersection line of the planes. It will soon be clear why we wish to express $\alpha_T([v_0,v_1])$ in this way, but first we aim to construct these vectors $n\in \operatorname{aff}\{v_0,v_1,v_2\}, m\in \operatorname{aff}\{v_0,v_1,v_3\}$ that are both orthogonal to $L=\operatorname{aff}\{v_0,v_1,v_2\} \cap \operatorname{aff}\{v_0,v_1,v_3\}$.
	
	Define $n$ as follows:
	\[
	n=(0,1,0)\in \operatorname{aff}\{v_0,v_1,v_2\}
	\]  
	Clearly, $n \perp L$ because $L$ is the $x$ axis.
	
	Next, let $\pi_L(v_3) = (x_3,0,0) \in L$, the orthogonal projection of $v_3$ onto the line $L$. By the properties of orthogonal projection,
	\begin{equation*}
		v_3-\pi_L(v_3) \perp L.
	\end{equation*}
	Moreover, $\pi_L(v_3)\in L$ means that $\pi_L(v_3) =tv_1$ for some $t\in \mathbb{R}$. Then, 
	\[
	v_3-\pi_L(v_3) = v_3-tv_1\in \operatorname{span}\{v_1,v_3\}=\operatorname{aff}\{v_0,v_1,v_3\}
	\]
	So, if we let $m= v_3-\pi_L(v_3)$ we have our desired vectors.
	Thus, $\theta_1$ is given by the angle between $n$ and $m$.
	Because $\theta_1$ is non-obtuse, it follows that,
	\begin{equation*}
		0\le n\cdot m = (0,1,0)\cdot(0,y_3,1) = y_3.
	\end{equation*}
	Therefore, $\pi_F(v_3) = (x_3,y_3,0) \in H_1^+$. Repeatedly applying this argument to the remaining     two edges of $F$ yields $\pi_F(v_3) \in H_1^+\cap H_2^+ \cap H_3^+ =F$. Therefore $F\subset T$ is a foundation and $T$ is a height function.

\end{proof}

\section{Existence of Height Function Decompositions}\label{BISEC}

The dihedral angle classification and Lemma \ref{L1} proved in the previous section allows us to prove the first primary result.
\begin{thm}\label{THM2}
	Let $T$ be a tetrahedron. Either $T$ is a height function, or, at most one tetrahedral subdivision of $T$ is needed to produce two height function tetrahedra $T_0$ and $T_1$.
\end{thm}

\begin{proof}
	
	As shown before, all tetrahedra fall within 4 distinct classes, some with multiple subclasses. As such, our strategy is to produce a proof for each class of tetrahedra. As it turns out, there are only 2 kinds of tetrahedra that require planar bisections to be decomposed into height functions, while all other kinds of tetrahedra are already height functions themselves. Let us first examine the fairly trivial cases of $T$ for which the initial configuration of $T$ is already a height function.
	
	\bigskip
	
	\noindent\textbf{Class $0$ tetrahedron:} Every Class 0 tetrahedron $T$ is a height function over all of its faces and no planar bisection is required. This result can be obtained by applying Lemma \ref{L1} to any face of $T$. 
	
	\bigskip
	
	\noindent\textbf{Class $1$ tetrahedron:} Let $e_0$ be the only edge with an obtuse dihedral angle. Now, consider the face $F$ with edges $e_1,e_2,e_3$ all not equal to $e_0$.\footnote{For this case, there exists two faces that satisfy this definition of $F$.} Because $e_0
	$ is the only edge with an obtuse dihedral angle, it must follow that $e_1,e_2,e_3$ all have non-obtuse dihedral angles. Then by Lemma \ref{L1}, $T$ is a height function relative to $F$, and no planar bisection is required. 
	
	\bigskip
	
	\noindent\textbf{Class 2 (a) tetrahedron:} Let $e_0$ and $e_1$ be the only edges of $T$ with obtuse dihedral angles. Since $T$ is of Class 2 (a), $e_0$ and $e_1$ share a common vertex, without loss of generality we can assume this vertex is $v_0$. Then consider the face $F=\operatorname{conv}\{v_1,v_2,v_3\}$ formed by the remaining three vertices. None of the three edges of $F$ contain $v_0$, hence every edge of $F$ must have a non-obtuse dihedral angle. Applying Lemma \ref{L1} shows that $F$ is a foundation. Therefore, $T$ is a height function and no planar bisection is required. 
	
	\bigskip
	
	\noindent\textbf{Class 3 (a) tetrahedron:} Let $e_0,e_1,e_2$ be the only 3 edges with obtuse dihedral angles. Since $T$ is of Class 3 (a), it must be the case that all these edges meet at a single vertex, after possibly relabeling, assume that this vertex is $v_0$. Then consider the face $F = \operatorname{conv}\{v_1,v_2,v_3\}$ formed by the remaining three vertices. None of the three edges of $F$ contain $v_0$, hence every edge of $F$ must have a non-obtuse dihedral angle. Applying Lemma \ref{L1} shows that $F$ is a foundation. Therefore, $T$ is a height function and no planar bisection is required. 
	
	\bigskip
	
	We now turn to the nontrivial cases, both of which require a planar subdivision. We first introduce a common geometric construction for $T$, and then specialize this construction according to the hypotheses of each subclass. Let $\{\Pi_p\}_{p\in [v_0,v_1]}$ be the one parameter family of affine planes that induces the family of tetrahedral subdivisions $\{T_0^p,T_1^p\}_{p \in[v_0,v_1]}$ on $T$. By construction, each sub-tetrahedron can be described by
	\begin{equation*}
		T_0^p = \operatorname{conv}\{v_0,v_2,v_3,p\}, \quad T_1^p = \operatorname{conv}\{v_1,v_2,v_3,p\}
	\end{equation*}
	and likewise, each shared face can be described by
	\begin{equation*}
		T \cap \Pi_p = \operatorname{conv} \{v_2,v_3,p\} \implies \bigcap_{p \in [v_0,v_1]} (T\cap \Pi_p) = [v_2,v_3]
	\end{equation*}
	Furthermore, the edges that are determined by $p$ are $[v_2,p]$ and $[v_3,p]$.
	
	\bigskip

	\noindent\textbf{Class 2 (b) tetrahedron:} Assume $T$ is a Class 2 Sublcass (b) tetrahedron. Notice that all tetrahedra of Class 2 (b) are configured so that the edges with obtuse dihedral angles are opposite from one another. Without loss of generality, assume $[v_0,v_1], [v_2,v_3]$ are the two opposite edges with obtuse dihedral angles. The arrangement of these edges requires that every face of $T$ contains at lease one of these edges. Since these edges have obtuse dihedral angles, Lemma \ref{L1} shows that no face of $T$ is a foundation, hence no Class 2 (b) tetrahedron is a height function. 
	
	Proceeding with the argument, using the identity $T_0^{v_1}=T_1^{v_0} = T$, we immediately deduce that
	\begin{equation*}
		\alpha_0([v_2,v_3],v_1) = \alpha_T([v_2,v_3]) > \frac{\pi}{2}, \quad \alpha_1([v_2,v_3],v_0) = \alpha_T([v_2,v_3]) > \frac{\pi}{2}
	\end{equation*}
	Manipulating Equation \ref{ALG2} relative to $\alpha_0([v_2,v_3],p)$ yields
	\begin{equation*}
		\alpha_0([v_2,v_3],p) = \alpha_T([v_2,v_3])-\alpha_1([v_2,v_3),p), \quad \forall p \in \operatorname{relint}([v_0,v_1])
	\end{equation*}
	Since $T_0^p$ degenerates at $v_0$, the function $\alpha_0([v_2,v_3],p)$ is undefined at $v_0$. Then, using the right hand limit, we find that
	\begin{align*}
		\lim_{p\to v_0^+}\alpha_0([v_2,v_3],p)
		&=\lim_{p\to v_0^+}
		\bigl(\alpha_T([v_2,v_3])-\alpha_1([v_2,v_3],p)\bigr)\\
		&=0<\frac{\pi}{2}.
	\end{align*}
	
	\begin{figure}
		\centering
		\includegraphics[scale=.83]{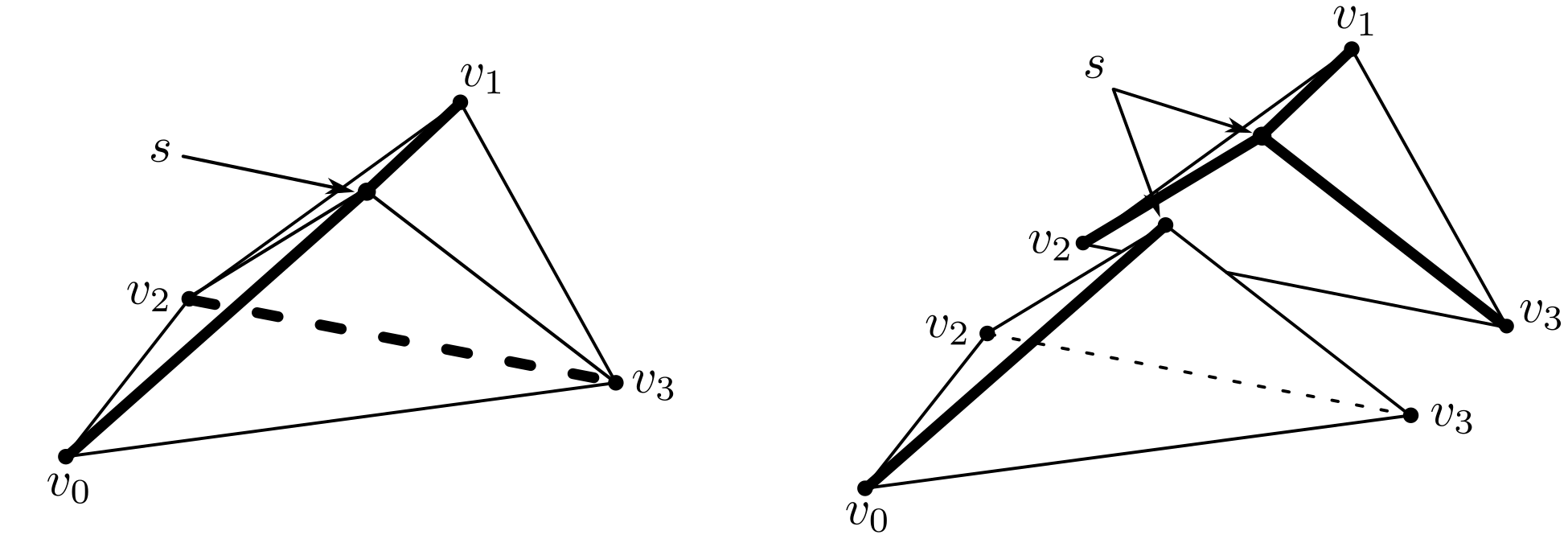}
		\caption{Height function tetrahedral subdivision of a Class 2 (b) tetrahedron with obtuse dihedral angles indicated by thickened lines (\textit{Left}) Original tetrahedron with proposed subdivision using point $s \in \operatorname{relint}([v_0,v_1])$.(\textit{Right}) Two height function tetrahedra separated after subdivision.}
		\label{S2}
	\end{figure}
	Then, by continuity of $\alpha_0([v_2,v_3],p)$ along with the Intermediate Value Theorem yields some $s \in \operatorname{relint}([v_0,v_1])$ such that $\alpha_0([v_2,v_3],s) = \frac{\pi}{2}$. Then once more, by the dihedral angle function relationship,
	\begin{align*}
		\alpha_1([v_2,v_3],s)
		&=\alpha_T([v_2,v_3])-\alpha_0([v_2,v_3],s)\\
		&=\alpha_T([v_2,v_3])-\frac{\pi}{2}
		<\frac{\pi}{2}.
	\end{align*}
	Then let us consider the affine plane $\Pi_s = \operatorname{aff}\{v_2,v_3,s\} \in \{\Pi_p\}_{p \in [v_0,v_1]}$ and its associated tetrahedral subdivision $T = T_0^s\cup T_1^s$. By construction, the two sub tetrahedron are
	\begin{equation*}
		T_0^s = \operatorname{conv}\{v_0,v_2,v_3,s\}, \quad T_1^s = \operatorname{conv}\{v_1,v_2,v_3,s\}
	\end{equation*}
	and the face they share is
	\begin{equation*}
		T_0^s \cap T_1^s = \operatorname{conv}\{v_2,v_3,s\}
	\end{equation*}
	First, the face $F = \operatorname{conv}\{v_0,v_2,v_3\} \subset T_0^s$ has edges $[v_0,v_2], [v_0,v_3]$ and $[v_2,v_3]$. By the choice of $s$ we have $\alpha_0([v_2,v_3],s) = \frac{\pi}{2}$. Notice that the dihedral angles of other 2 edges remain unchanged by the decomposition, and, by hypothesis, have non-obtuse dihedral angles in $T_0^s$ inherited by $T$. An application of Lemma \ref{L1} shows that F is a foundation of $T_0^s$. Next, consider the face $K = \operatorname{conv}\{v_1,v_2,v_3\}\subset T_1^s$ with edges $[v_1,v_2],[v_1,v_3]$ and $[v_2,v_3]$. We know that $\alpha_1([v_2,v_3],s) <\frac{\pi}{2}$ and the dihedral angles of the other 2 edges are unchanged, so they must have non-obtuse dihedral angles in $T_1^s$. Thus, by Lemma \ref{L1}, $K$ is a foundation for $T_1^s$. Therefore, $T_0^s$ and $T_1^s$ are both height function tetrahedra relative to their respective foundation faces.

	\bigskip
	
	\noindent\textbf{Class 3 (b) tetrahedron:} Assume $T$ is a Class 3 (b) tetrahedra. The configuration constitutes that $T$ has 3 edges with obtuse dihedral angles that do not meet at a common vertex. Without loss of generality, we may assume that $[v_0,v_1], [v_1,v_2]$ and $[v_2,v_3]$ are the edges where obtuse dihedral angles occur. By Theorem \ref{T1}, the remaining edges $[v_0,v_2],[v_0,v_3],[v_1,v_3]$ all have acute dihedral angles.

	To begin, we investigate the behavior of the dihedral angle functions $\alpha_0([v_2,p],p)$ and $\alpha_1([v_2,p],p)$ with respect to the edge $[v_2,p]$ determined as $p$ varies. We claim that, for all $p \in \operatorname{relint}([v_0,v_1])$ the angle function, $\alpha_0([v_2,p],p)$ satisfies
	\begin{equation*}
		\alpha_0([v_2,p],p) > \frac{\pi}{2}
	\end{equation*}
	and subsequently, for all $p\in \operatorname{relint}([v_0,v_1])$
	\begin{equation*}
		\alpha_1([v_2,p],p) < \frac{\pi}{2}
	\end{equation*}
	We start by realizing that the identity $T_0^{v_1} = T_1^{v_0}=T$ allows us to deduce that
	\begin{equation*}
		\alpha_0([v_2,v_1],v_1) = \alpha_T([v_1,v_2]) >\frac{\pi}{2},\quad \alpha_1([v_2,v_0],v_0) = \alpha_T([v_0,v_2]) <\frac{\pi}{2}
	\end{equation*}
	By a slight manipulation of Equation \ref{ALG2} for $\alpha_0([v_2,p],p)$, we have that
	\begin{equation*}
		\alpha_0([v_2,p],p) = \pi- \alpha_1([v_2,p],p), \quad \forall p \in \operatorname{relint}([v_0,v_1])
	\end{equation*}
	Then, since $T_0^p$ degenerates at $v_0$, the function $\alpha_0([v_2,p],p)$ is undefined at $v_0$. Using the right hand limit and the fact that $\alpha_T([v_0,v_2]) < \frac{\pi}{2}$, we have that
	\begin{align*}
		\lim_{p\to v_0^+}\alpha_0([v_2,p],p) &=\lim_{p \to v_0^+}( \pi - \alpha_1([v_2,p],p]) \\
		&= \pi - \alpha_T([v_0,v_2]) > \frac{\pi}{2}
	\end{align*}
	Thus, $\alpha_0([v_2,p],p)$ is greater than $\frac{\pi}{2}$ at the both vertex endpoints $v_0$ and $v_1$. Then, because $\alpha_0([v_2,p],p)$ is monotonic on $\operatorname{relint}([v_0,v_1])$, we have that 
	\begin{equation*}
		\alpha_0([v_2,p],p) > \frac{\pi}{2} \quad \forall p \in \operatorname{relint}([v_0,v_1])
	\end{equation*}
	\begin{figure}
		\centering
		\includegraphics[scale=.83]{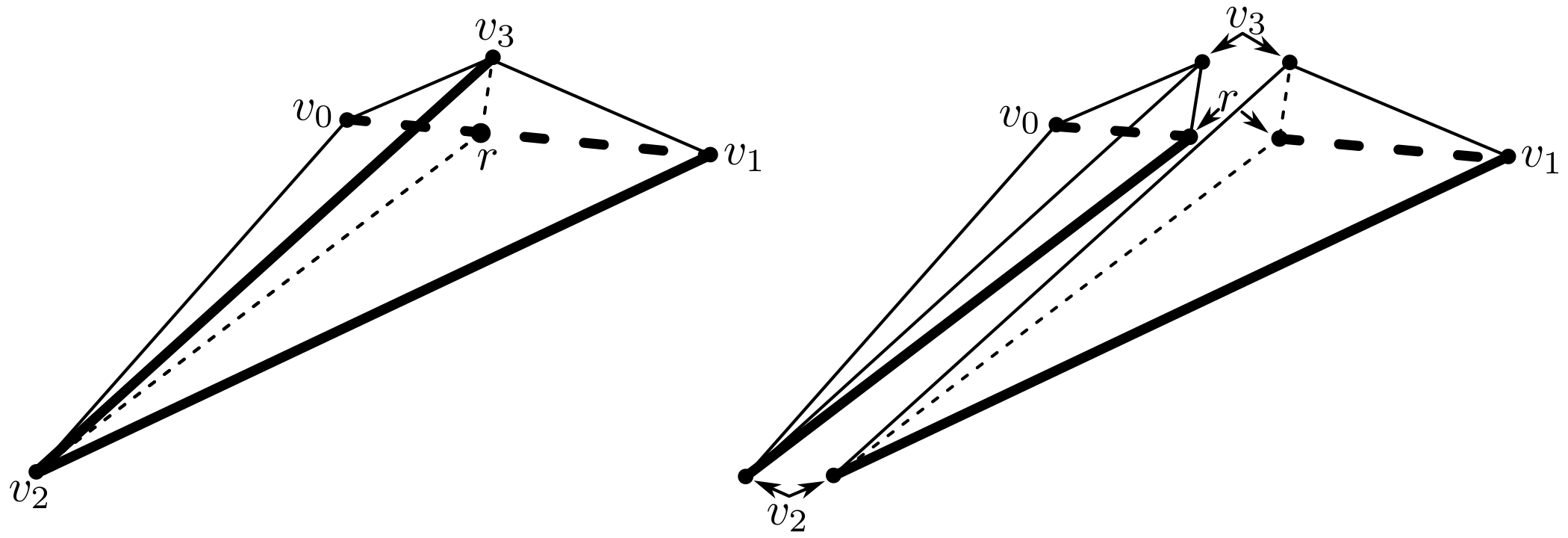}
		\caption{Height function tetrahedral subdivision of a Class 3 (b) tetrahedron with obtuse dihedral angles indicated by thickened lines (\textit{Left}) Original tetrahedron with proposed subdivision using point $r \in \operatorname{relint}([v_0,v_1])$.(\textit{Right}) Two height function tetrahedra separated after subdivision.}
		\label{S3}
	\end{figure}
	and subsequently, by the supplementary angle relationship,
	\begin{equation*}
		\alpha_1([v_2,p],p) < \frac{\pi}{2} \quad \forall p \in \operatorname{relint}([v_0,v_1])
	\end{equation*}
	Now, we examine the behavior of the angle functions $\alpha_0([v_3,p],p)$ and $\alpha_1([v_3,p],p)$ with respect to the edge $[v_3,p]$. Again, using the identity $T_0^{v_1}=T_1^{v_0} = T$, we deduce that
	\begin{equation*}
		\alpha_0([v_3,v_1],v_1) = \alpha_T([v_1,v_3])< \frac{\pi}{2}, \quad \alpha_1([v_3,v_0],v_0) = \alpha_T([v_0,v_3]) <\frac{\pi}{2}
	\end{equation*}
	Once more, a slight manipulation of Equation \ref{ALG2} relative to $\alpha_0([v_3,p],p)$ yields
	\begin{equation*}
		\alpha_0([v_3,p],p) = \pi - \alpha_1([v_3,p],p) \quad \forall p \in \operatorname{relint}([v_0,v_1]
	\end{equation*}
	Then, since $T_0^p$ degenerates at $v_0$, the function $\alpha_0([v_3,p],p)$ is undefined at $v_0$. Then using the right hand limit yields
	\begin{align*}
		\lim_{p\to v_0^+}\alpha_0([v_3,p],p) &= \lim_{p\to v_0^+}(\pi - \alpha_1([v_3,p],p)) \\ &= \pi - \alpha_T([v_0,v_3]) > \frac{\pi}{2}
	\end{align*}
	Now, since $\alpha_0([v_3,p],p)$ is continuous on $\operatorname{relint}([v_0,v_1])$, the Intermediate Value Theorem shows that there exists some $r\in \operatorname{relint}([v_0,v_1])$ such that 
	\begin{equation*}
		\alpha_0([v_3,r],r) = \frac{\pi}{2} \quad \text{ and } \quad \alpha_1([v_3,r],r) = \frac{\pi}{2}
	\end{equation*}
	Fix $r \in \operatorname{relint}([v_0,v_1])$ and define the affine plane $\Pi_r = \operatorname{aff}\{v_2,v_3,r\} \in \{\Pi_p\}_{p \in [v_0,v_1]}$ and the two tetrahedron $T_0^r ,T_1^r \in \{T_0^p,T_1^P\}_{p \in [v_0,v_1]}$ produced by the subdivision by $\Pi_r$. By construction,
	\begin{equation*}
		T_0^r = \operatorname{conv}\{v_0,v_2,v_3,r\}, \quad T_1^r = \operatorname{conv}\{v_1,v_2,v_3,r\}, \quad
		T_0^r \cap T_1^r = \operatorname{conv}\{v_2,v_3,r\}
	\end{equation*}
	We claim that $F=\operatorname{conv} \{v_0,v_2,v_3\} \subset T_0^r$ and $K=\operatorname{conv}\{v_2,v_3,r\}\subset T_1^r$ are both foundations for their respective tetrahedron. Beginning with $F$, we consider each edge $[v_0,v_2],$ $[v_0,v_3]$, and $[v_2,v_3]$ of $F$. By hypothesis, the dihedral angles that occur at $[v_0,v_2]$ and $[v_0,v_3]$ are acute, and remain unchanged by the subdivision. Thus, it remains to check the dihedral angle at $[v_2,v_3]$ is non-obtuse. To verify that the dihedral angle at $[v_2,v_3]$ is in fact non-obtuse, we'll use Theorem \ref{T1} and the proven behavior of the angle function $\alpha_0$ with respect to $[v_2,r],[v_3,r] \in E(T_0^r)$. By our selection of $r \in \operatorname{relint}([v_0,v_1]$ we have that
	\begin{equation*}
		\alpha_0([v_2,r],r)>\frac{\pi}{2} \quad \text{ and } \quad \alpha_0([v_3,r],r) = \frac{\pi}{2}
	\end{equation*}
	Since the edge $[v_0,r] \in E(T_0^r)$ inherits the obtuse dihedral angle of $[v_0,v_1]$, the only acute dihedral angles of $T_0^r$ occur at the edges $[v_0,v_2]$ and $[v_0,v_3]$. Thus, since all other edges of $T_0^r$ are non-acute, by Theorem \ref{T1}, the dihedral angle at $[v_2,v_3]$ must be acute. A direct application of Lemma \ref{L1} shows that $F$ is a foundation of $T_0^r$, and therefore, $T_0^r$ is a height function tetrahedron.

	Next, we show $T_1^r$ is a height function.  Consider the edges $[v_1,v_2], [v_1,r],[v_3,r]\in E(T_1^r)$. By our choice of $r$ we have $\alpha_1([v_3,r],r) = \frac{\pi}{2}$. Moreover, the other 2 edges inherit obtuse angles from $T$. Hence, all three edges have non-acute dihedral angles, so by Theorem \ref{T1} we know the 3 remaining edges of $T_1^r$ including $[v_2,r]$ and $[v_2,v_3]$ must have acute dihedral angles in $T_1^r$. Moreover, $[v_3,r]$ is non-obtuse, and thus an application of Lemma \ref{L1} to the face $K=\operatorname{conv}\{v_2,v_3,r\}\subset T_1^r$ shows that $K$ is a foundation and $T_1^r$ is a height function. Therefore $T_0^r$ and $T_1^r$ are both height functions.
\end{proof}

While developing a proof for Theorem \ref{THM2}, a primary motivator for the selection of edges to perform a planar bisection along was based on the need to reduce the number of obtuse dihedral angles present with the expectation of then being able to directly apply Lemma \ref{L1} to the subsequent sub-tetrahedra. Observe that, for both Class 3 (b) and Class 2 (b), the affine plane used to induce a tetrahedral subdivision contains an edge whose dihedral angle is obtuse, and passes through a single point contained within the opposite edge whose dihedral angle is also obtuse. Naturally, one may immediately question whether such a phenomenon is purely coincidence or rather a result of underlying structure. After investigating other potential cuts, we find that the latter is in fact true.

\begin{thm}\label{UNIQ}
	Let $T$ be a tetrahedron of Class $n \ge 2$, (b). If an affine plane $\Pi$ induces a height function tetrahedral subdivision of $T$, then the edge $e$ contained in $\Pi$ and the opposite edge $e'$ intersected by $\Pi$ at a point $p \in \operatorname{relint}(e')$ must both have obtuse dihedral angles.
\end{thm}

Prior to reading the following proof, we recommend that the reader revisit Lemma \ref{L1}, Theorem \ref{T1}, and Lemma \ref{TS}. 

\begin{proof}
	
	\textbf{Class 2 (b):} Let $T=\operatorname{conv}\{v_0,v_1,v_2,v_3\}$ be a Class 2 Sublcass (b) tetrahedron and define the family of affine planes $\{\Pi_p\}_{p \in [v_0,v_1]}$ on such that each planes $\Pi_p$ contains the point $p$ and the edge $[v_2,v_3]$. Suppose for contradiction that the dihedral angles measured at $[v_0,v_1]$ and $[v_2,v_3]$ are not both obtuse. Then at least $[v_0,v_1]$ or $[v_2,v_3]$ has a non-obtuse dihedral angle. However, observe that, because $T$ is a Class 2 (b) tetrahedron, the only obtuse dihedral angles of $T$ must occur at edges opposite to each other, thus forcing both dihedral angles measured at $[v_0,v_1]$ and $[v_2,v_3]$ to both be non-obtuse. Then without loss of generality, assume the dihedral angles measured at $[v_0,v_2]$ and $[v_1,v_3]$ are obtuse as selecting the other pair of edges follows an identical argument. 
	
	First, by the identity $T_0^{v_1}=T_1^{v_0}=T$, we have that
	\begin{equation*}
		\alpha_0([v_3,v_1],v_1) = \alpha_T([v_3,v_1])> \frac{\pi}{2}, \quad \alpha_1([v_3,v_0],v_0) = \alpha_T([v_0,v_3]) \le \frac{\pi}{2}
	\end{equation*}
	Since $\alpha_0([v_3,p],p)$ and $\alpha_1([v_3,p],p)$ share a supplementary relationship, and $\alpha_0([v_3,p],p)$ is undefined at $v_0$, we use the right hand limit to deduce that
	\begin{align*}
		\lim_{p \to v_0^+}\alpha_0([v_3,p],p)&= \lim_{p \to v_0^+}(\pi-\alpha_1([v_3,p],p) \\
		&= \pi - \alpha_T([v_0,v_3])\ge\frac{\pi}{2}
	\end{align*}
	Because $\alpha_0([v_3,p],p)$ is monotonic on $p \in \operatorname{relint}([v_0,v_1])$, it follows that
	\begin{equation*}
		\alpha_0([v_3,p],p) > \frac{\pi}{2}, \quad \forall p \in \operatorname{relint}([v_0,v_1])
	\end{equation*}
	Now, fix $l \in \operatorname{relint}([v_0,v_1])$, and consider the plane $\Pi_l = \operatorname{conv}\{v_2,v_3,l\} \in \{\Pi_p\}_{p \in[v_0,v_1]}$ that induces the subdivision
	\begin{equation*}
		T_0^l = \operatorname{conv}\{v_0,v_2,v_3,l\}, \quad T_1^l = \operatorname{conv}\{v_1,v_2,v_3,l\}, \quad T_0^l \cap T_1^l = \operatorname{conv}\{v_2,v_3,l\}
	\end{equation*}
	Observe that, for any $p \in \operatorname{relint}([v_0,v_1])$, the obtuse dihedral angle measured at the edge $[v_0,v_2] \subset T_0^l$ is unchanged by the decomposition, and thus, remains obtuse in $T_0^l$. Then since $[v_0,v_2]$ and $[v_3,l]$ are opposite to each other and both have obtuse dihedral angles in $T_0^l$, it must follow that $T_0^l$ is of Class 2 (b). As proven before in Theorem \ref{THM2}, $T_0^l$ cannot be a height function, a contradiction to our assumption. Thus, $\alpha_T([v_0,v_1])$ and $\alpha_T([v_2,v_3])$ must be obtuse.
	
	\bigskip
	
	\noindent\textbf{Class 3 (b):} Assume $T$ is a Class 3 (b) tetrahedron. Let $q\in\operatorname{relint}([v_0,v_1])$ be such that the affine plane $\Pi_q$ induces the height function tetrahedral subdivision
	\begin{equation*}
		T_0^q=\operatorname{conv}\{v_0,v_2,v_3,q\},\qquad
		T_1^q=\operatorname{conv}\{v_1,v_2,v_3,q\}.
	\end{equation*}
	To derive a contradiction, we analyze the one-parameter family of subdivisions $T_1^q,T_0^q \in \{T_0^p,T_1^p\}_{p \in [v_0,v_1]}$ induced by the family of planes $\Pi_q \in \{\Pi_{p}\}_{p \in [v_0,v_1]}$. We consider 3 arguments based on the arrangements of obtuse dihedral angles seen in Figure \ref{fig:CLASS3BK4}.
	
	\begin{figure}[h!]
		\centering
		\includegraphics[scale=.90]{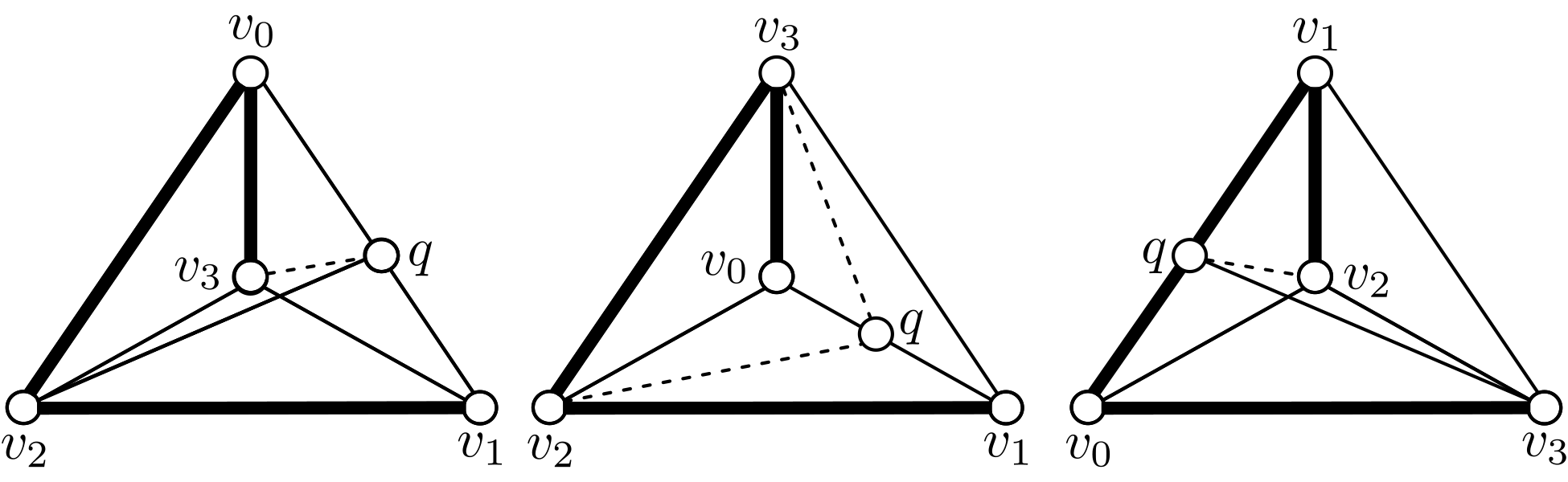}
		\caption{$K_4$ representations of Class 3 (b) tetrahedra not where opposite edges are not both obtuse.}
		\label{fig:CLASS3BK4}
	\end{figure}

	\bigskip
	
	\noindent\textbf{Case 1:} Without loss of generality, assume for purpose of contradiction that $[v_1,v_2]$, $[v_2,v_3]$ and $[v_0,v_3]$ have obtuse dihedral angles. Then by Theorem \ref{T1}, the remaining edges $[v_0,v_2]$, $[v_1,v_3],$ and $[v_0,v_1]$ must have acute dihedral angles. Now, consider the edge $[v_3,p]$ formed by an arbitrary subdivision of $T$. By the identity $T_0^{v_1}=T_1^{v_0}=T$, we have that
	\begin{equation*}
		\alpha_0([v_3,v_1],v_1) = \alpha_T([v_1,v_3]) < \frac{\pi}{2}, \quad \alpha_1([v_3,v_0],v_0) = \alpha_T([v_0,v_3]) > \frac{\pi}{2}
	\end{equation*}
	Using Equation \ref{ALG2} and the left hand limit, we have that
	\begin{align*}
		\lim_{p \to v_1^-} \alpha_1([v_3,p],p) &= \lim_{p \to v_1^-}(\pi - \alpha_0([v_3,p],p) \\
		&= \pi - \alpha_T([v_1,v_3]) > \frac{\pi}{2}
	\end{align*}

	Since $\alpha_1([v_3,p],p)$ is monotonic on $\operatorname{relint}([v_0,v_1])$, it must follow that 
	\begin{equation*}
		\alpha_1([v_3,p],p) > \frac{\pi}{2},\quad \forall p \in \operatorname{relint}([v_0,v_1])
	\end{equation*}
	Taking $p =q \in \operatorname{relint}([v_0,v_1])$, since $[v_1,v_2]$, by hypothesis, has an obtuse dihedral angle, $[v_3,q]$ always has an obtuse dihedral angle in $T_1^q$, and $[v_1,v_2]$ and $[v_3,p]$ are opposite, $T_1^q$ must be of Class 2 (b). Therefore, $T_1^q$ is not a height function, a contradiction.
	
	\bigskip
	
	\noindent\textbf{Case 2:} Assume for purpose of contradiction that $[v_1,v_2]$,$[v_0,v_2]$, and $[v_0,v_3]$ have obtuse dihedral angles. By Theorem \ref{T1}, the remaining edges $[v_2,v_3]$, $[v_0,v_1]$ and $[v_1,v_3]$ all have acute dihedral angles. consider the edge $[v_3,p]$ formed by an arbitrary subdivision of $T$. Consider the edge $[v_3,p]$. Notice that, by the identity $T_0^{v_1}=T_1^{v_0} = T$, we have that 
	\begin{equation*}
		\alpha_0([v_3,v_1],v_1) = \alpha_T([v_1,v_3]) < \frac{\pi}{2}, \quad \alpha_1([v_3,v_0],v_0) = \alpha_T([v_0,v_3]) > \frac{\pi}{2}
	\end{equation*}
	This configuration is identical to Case 1, allowing us to conclude that
	\begin{equation*}
		\alpha_1([v_3,p],p) > \frac{\pi}{2},\quad \forall p \in \operatorname{relint}([v_0,v_1])
	\end{equation*}
	Once more, taking $p=q \in \operatorname{relint}([v_0,v_1])$, since $[v_1,v_2]$, by hypothesis, has an obtuse dihedral angle, $[v_3,q]$ always has an obtuse dihedral angle in $T_1^q$, and $[v_1,v_2]$ and $[v_3,q]$ are opposite, $T_1^q$ must be of Class 2 (b). Therefore, $T_1^q$ is not a height function, a contradiction.
	
	\begin{remark}
		The argument for the plane that contains the edge $[v_0,v_1]$ and intersects $q\in \operatorname{relint}([v_2,v_3])$ is identical via symmetry.
	\end{remark}
	
	\bigskip
	
	\noindent\textbf{Case 3:} Assume for purpose of contradiction that the edges $[v_0,v_1],[v_1,v_2]$ and $[v_0,v_3]$ have obtuse dihedral angles. By Theorem \ref{T1}, the remaining edges $[v_0,v_2],[v_2,v_3]$ and $[v_1,v_3]$ all have acute dihedral angles. Consider the edge $[v_3,p] \in E(T_0^p) \cap E(T_1^p)$. Notice one more, by the identity $T_0^{v_1} = T_1^{v_0} = T$, we have that
	\begin{equation*}
		\alpha_0([v_3,v_1],v_1) = \alpha_T([v_1,v_3]) < \frac{\pi}{2}, \quad \alpha_1([v_3,v_0],v_0) = \alpha_T([v_0,v_3]) > \frac{\pi}{2}
	\end{equation*}
	Thus, we have an identical configuration to Case 1 and Case 2, allowing us to conclude that
	\begin{equation*}
		\alpha_1([v_3,p],p) > \frac{\pi}{2},\quad \forall p \in \operatorname{relint}([v_0,v_1])
	\end{equation*}
	Fix $p=q \in \operatorname{relint}([v_0,v_1])$. Consider the tetrahedron $T_1^q = \operatorname{conv}\{v_1,v_2,v_3,q\}$ formed by the subdivision. Since $[v_1,v_2]$ and $[v_1,q] \subset[v_0,v_1]$ have obtuse dihedral angles unchanged by the subdivision of $T$, and $[v_3,q]$ always has an obtuse dihedral angle in $T_1^q$, we have that $[v_1,v_2],[v_1,q],$ and $[v_3,q]$ form a path of length 3. Thus, $T_1^q$ is a Class 3 (b) tetrahedron, and therefore not a height function tetrahedron, a contradiction.
	
\end{proof}

\section{Families of Height Function Inducing Affine Planes}\label{SEC5}
Throughout this study, we have analyzed the behavior of dihedral angle functions on general parametrized families of tetrahedral subdivisions in relation to dihedral angle classifications. This analysis has established both the feasibility of height function subdivisions (Theorem \ref{THM2}) and the rigid geometric structures governing the families of affine planes that induce them (Theorem \ref{UNIQ}). By exploiting continuity and monotonicity properties of the associated dihedral angle functions, we have shown the existence of specific parameter values that produce height function subdivisions. However, a natural question remains: are these parameter values the only ones capable of inducing such subdivisions? Motivated by this question, we establish one final result that, when unified with the prior Theorems, provides a full classification of height function subdivisions and the affine plane families to which they belong.

\begin{thm}
	Let $T$ be a tetrahedron with $e,e'\in E(T)$ opposite edges that have obtuse dihedral angles. Then there exists non-trivial intervals $I\subset \operatorname{relint}(e)$ and $I'\subset \operatorname{relint}(e')$ such that the affine plane $\Pi_p$
	induces a height function tetrahedral subdivision if and only if $p\in I$ or $p\in I'$ respectively.
	
\end{thm}

\begin{proof}
	It suffices to prove the existence of only one of these intervals as we can simply relabel vertices and applying the same argument to the opposite edge. The proof will follow the same format of proving the statement for  Class 2 (b) and Class 3 (b) tetrahedra independently.
	
	\bigskip
	
	\noindent\textbf{Class 2 (b):} First, assume $T$ is a Class 2 Sublcass (b) tetrahedron. Without loss of generality we can assume that $[v_0,v_1]$ and $[v_2,v_3]$ are the edges where obtuse dihedral angles occur. First, note that in the proof of Theorem \ref{THM2} for Class 2 Sublcass (b) tetrahedra, we proved the existence of a point $s\in \operatorname{relint}([v_0,v_1])$ that satisfied:
	\[
	\alpha_0([v_2,v_3],s) = \frac{\pi}{2} \quad \text{ and } \quad \alpha_1([v_2,v_3],s) < \frac{\pi}{2}
	\]
	An identical argument can be used to produce a point $t\in \operatorname{relint}([v_0,v_1])$ that satisfies:
	\[
	\alpha_0([v_2,v_3],t) < \frac{\pi}{2} \quad \text{ and } \quad \alpha_1([v_2,v_3],t) = \frac{\pi}{2}
	\]
	Since $\alpha_0([v_2,v_3],p)$ is increasing on $\operatorname{relint}([v_0,v_1])$ and $\alpha_0([v_2,v_3],t)<\alpha_0([v_2,v_3],s)$ we know that $t<s$. Therefore $[t,s]$ defines a non-trivial interval contained within $\operatorname{relint}([v_0,v_1])$. We claim that $[t,s]$ is our desired interval.
	
	We first show that any point in $[t,s]$ produces a height function tetrahedral subdivision. Let $k\in \operatorname{relint}([t,s])$ and consider the following faces,
	\[
	F = \operatorname{conv}\{v_0,v_2,v_3\}\subset T_0^k \quad \text{ and } \quad K = \operatorname{conv}\{v_1,v_2,v_3\}\subset T_1^k
	\]
	By the increasing and decreasing nature of these functions,
	\[
	\alpha_0([v_2,v_3],k) < \alpha_0([v_2,v_3],s) = \frac{\pi}{2} \quad \text{ and } \quad \alpha_1([v_2,v_3],k) < \alpha_1([v_2,v_3],t) = \frac{\pi}{2}
	\]
	Moreover, the dihedral angles of the remaining edges of $F$ and $K$ are unchanged by the decomposition. Thus,
	\[
	\alpha_0([v_0,v_3],k), \alpha_0([v_0,v_2],k) \le \frac{\pi}{2},\quad  \alpha_1([v_1,v_2],k), \alpha_1([v_1,v_3],k) \le \frac{\pi}{2}
	\]
	Therefore, by Lemma \ref{L1}, $F$ and $K$ are both foundations, hence $T_0^k$ and $T_1^k$ are height functions.
	
	Next, we show that any point outside this interval does not induce a height function tetrahedral subdivision. Suppose $ k\in [v_0,t)\cup (s,v_1]$. Then either $k \in [v_0,t)$ or $k \in (s,v_1]$. For simplicity, we assume $k \in [v_0,t)$, as the other case follows a symmetrically identical argument. Then,
	\begin{equation*}   
		\alpha_1([v_2,v_3],k) > \alpha_1([v_2,v_3],t) = \frac{\pi}{2}
	\end{equation*} 
	Moreover, since $[v_0,v_1]$ has an obtuse dihedral angle in $T$ we know that $\alpha_1([k,v_1],k) > \frac{\pi}{2}$. Hence $[v_2,v_3],[k,v_1] \in E(T_1^k)$ are opposite edges with obtuse dihedral angles, meaning that $T_1^k$ is a Class 2 (b) tetrahedron and thus not a height function. Therefore, any $k \notin [t,s]$ fails to induces a height function tetrahedral subdivision of $T$. This shows $[t,s]\subset \operatorname{relint}([v_0,v_1])$ is our desired interval.
	
	\bigskip
	
	\noindent\textbf{Class 3 (b):} Suppose $T$ is a Class 3 (b) tetrahedron. Without loss of generality we assume that $[v_0,v_1], [v_1,v_2]$ and $[v_2,v_3]$ all have obtuse dihedral angles. First we prove some useful results:
	
	\begin{enumerate}
		\item $\alpha_1([v_2,p],p) < \frac{\pi}{2}$ for all $p\in[v_0,v_1]$
		\item $\alpha_1([v_3,p],p)$ is increasing on $[v_0,v_1]$

	\end{enumerate}
	We already showed (1) to be true in the proof of Theorem \ref{THM2}.  (2) follows from limit behavior. Indeed, using the fact that $T_1^{v_0} =T_0^{v_1}=T$ and the complementary angle relation alongside the left hand limit we get,
	\[
	\alpha_1([v_3,v_0],v_0)=\alpha_T(v_3,v_0) < \frac{\pi}{2}
	\] 
	and,
	\begin{align*}
		\lim_{p\to v_1^-}\alpha_1([v_3,p],p) &= \lim_{p\to v_1^-} (\pi- \alpha_{0}([v_3,p], p))\\ &= \pi - \alpha_T(v_3,v_1) > \frac{\pi}{2}
	\end{align*}
	Monotonicity on $\operatorname{relint}([v_0,v_1])$ then implies that $\alpha_1([v_3,p],p)$ is increasing.
	
	In the proof of Theorem \ref{THM2} we proved for all Class 3 (b) tetrahedra  with $[v_0,v_1], [v_1,v_2]$ and $[v_2,v_3]$ having obtuse dihedral edges, there exists some $s\in \operatorname{relint}([v_0,v_1])$ such that,
	\[
	\alpha_0([v_3,s],s)=\alpha_1([v_3,s],s) = \frac{\pi}{2}, \quad \alpha_1([v_2,v_3],s) < \frac{\pi}{2}
	\]
	We will use this $s$ as one of the endpoints for our desired interval. Next, we produce a $t\in \operatorname{relint}([v_0,v_1])$ distinct from $s$ that will act as a seperate endpoint. Because $T_1^{v_0} = T$,
	\[
	\alpha_1([v_2,v_3],v_0) = \alpha_T([v_2,v_3]) > \frac{\pi}{2}
	\]
	and
	\begin{align*}
		\lim_{p \to v_1^-}\alpha_1([v_2,v_3],p) &= \lim_{p \to v_1^-}(\pi - \alpha_0([v_2,v_3],p) \\
		&= \pi - \alpha_T([v_2,v_3]) < \frac{\pi}{2}
	\end{align*}
	Then, we may use the Intermediate Value Theorem to produce some $t\in\operatorname{relint}([v_0,v_1])$ with $\alpha_1([v_2,v_3],t) = \frac{\pi}{2}$.
	Since $\alpha_{1}([v_2,v_3],p)$ is decreasing and $\alpha_1([v_2,v_3],t) = \frac{\pi}{2} > \alpha_{1}([v_2,v_3],s)$ it follows that $t<s$ and $[t,s]$ is a well defined interval in $\operatorname{relint}([v_0,v_1])$. 
	
	Now we aim to show that any $k\in [t,s]$ induces a height function tetrahedral subdivision. Indeed, we claim that $F=\operatorname{conv}\{v_0,v_2,v_3\}\subset T_0^k$ and $K=\operatorname{conv}\{v_2,v_3,k\}\subset T_1^k$ are both foundation faces whenever $k\in [t,s]$. Because $\alpha_{0}([v_2,v_3],p)$ is increasing and $\alpha_{1}([v_2,v_3],p)$ is decreasing,
	\[
	\alpha_{0}([v_2,v_3],k) < \alpha_{0}([v_2,v_3],s) < \frac{\pi}{2} \quad \text{ and } \quad \alpha_{1}([v_2,v_3],k) < \alpha_{1}([v_2,v_3],t) = \frac{\pi}{2}
	\]
	The dihedral angles of $[v_0,v_2]$ and $[v_0,v_3]$ are unchanged by the decomposition, so
	\[
	\alpha_{0}([v_0,v_2],k), \alpha_0([v_0,v_3],k) \le \frac{\pi}{2}
	\]
	Lastly, because $\alpha_{1}([v_2,p],p) < \frac{\pi}{2}$ for all $p\in [v_0,v_1]$ and because $\alpha_{1}([v_3,p],p)$ is increasing we have,
	\[
	\alpha_{1}([v_2,k],k) < \frac{\pi}{2} \quad \text{ and } \quad \alpha_{1}([v_3,k],k) < \alpha_{1}([v_3,s],s) = \frac{\pi}{2}
	\]
	Therefore, by Lemma \ref{L1}, $F\subset T_0^k$ and $K\subset T_1^k$ are foundations, hence $T_0^k$ and $T_1^k$ are height functions.
	
	Finally, we show that any $k\in [v_0,t)\cup (s,v_1]$ does not yield a height function decomposition. First, suppose $k\in [v_0,t)$. Since $\alpha_{1}([v_2,v_3],p)$ is decreasing,
	\[
	\alpha_{1}([v_2,v_3],k) > \alpha_{1}([v_2,v_3],t) = \frac{\pi}{2}
	\]
	Furthermore, $[v_1,k]\subset T_1^k$ and $[v_1,v_2]\subset T_1^k$ inherit obtuse dihedral angles from $T$. This shows that $T_1^k$ is of Class 3 (b) and not a height function.
	
	Similarly, suppose $k\in (s,v_1]$, then because $\alpha_{1}([v_3,p],p)$ is increasing, 
	\[
	\alpha_{1}([v_3,k],k) > \alpha_{1}([v_3,s],s) = \frac{\pi}{2}
	\]
	Moreover, $[v_1,k]\subset T_1^k$ and $[v_1,v_2]\subset T_1^k$ inherit obtuse dihedral angles from $T$. Hence, $T_1^k$ is of Class 3 (b) and not a height function. This proves that $[t,s]$ is our desired interval.
\end{proof}

The culmination of Lemma \ref{TS}, Theorem 5.1 and Theorem 5.2 allows us to completely classify all possible height function tetrahedral subdivisions for Class 2 (b) and Class 3 (b) tetrahedra. To put it simply, the bisecting plane must contain an edge with an obtuse dihedral angle and the opposite edge the plane intersects must also have an obtuse dihedral angle. Moreover, the plane must intersect this opposite edge at exactly one point within a non-trivial interval.

\section{Conclusion and Conjecture}
In Section \ref{PRELIM}, we established the notion of a tetrahedral subdivision and for the remainder of the paper, we restricted our focus to planar bisections of this form. The driving reason behind this specification was ultimately to simplify our task by avoiding complex polyhedra. Although this restriction does not capture the theory of height functions in full generality, for the context of this paper, no significant sacrifices are made by narrowing our scope. In theory, decomposing more complex polyhedra will likely require a broader class of subdivisions. However, in the case of tetrahedra, we propose the following conjecture without proof. 

\bigskip

	\noindent\textbf{Conjecture.} \textit{Let $T$ be a tetrahedron of Class 2 (b) or Class 3 (b), and let the plane $\Pi$ induce a planar bisection $T = P_0\cup P_1$, where $P_0$ and $P_1$ are both height function polyhedra. Then this planar bisection must be a tetrahedral subdivision. }
	
\section*{Acknowledgments}
The problem for this paper was initially proposed by Dr. Maria Trnkova, and was made possible through the support of both Dr. Trnkova and the Davis Math Lab.

\end{document}